\documentclass{amsart}
\usepackage{amssymb,latexsym,amsmath,amscd,graphicx,graphics,epic,eepic,bm,color,array,mathrsfs,fullpage}
\usepackage{enumerate}
\usepackage[all,knot,poly]{xy}

\newcommand{\zed}{\mathbb{Z}}

\newcommand{\C}{\mathbb{C}}
\newcommand{\Q}{\mathbb{Q}}

\newcommand{\im}{\mathrm{Im}}
\newcommand{\Sym}{\mathrm{Sym}}

\newcommand{\ve}{\varepsilon}

\newcommand{\id}{\mathrm{id}}

\newcommand{\rank}{\mathrm{rank}}

\newcommand{\Kdim}{\mathrm{K}\text{-}\dim}
\newcommand{\ann}{\mathrm{ann}}
\newcommand{\codim}{\mathrm{co}\text{-}\dim}

\theoremstyle{plain}
\newtheorem{theorem}{Theorem}[section]
\newtheorem{lemma}[theorem]{Lemma}
\newtheorem{proposition}[theorem]{Proposition}
\newtheorem{corollary}[theorem]{Corollary}
\newtheorem{conjecture}[theorem]{Conjecture}
\newtheorem{question}[theorem]{Question}
\newtheorem{fact}[theorem]{Fact}

\theoremstyle{definition}
\newtheorem{definition}[theorem]{Definition}

\newtheorem{acknowledgments}{Acknowledgments\ignorespaces}

\theoremstyle{remark}
\newtheorem{remark}[theorem]{Remark}

\numberwithin{equation}{section}

\begin{document}

\title{Khovanov-Rozansky homology and Directed Cycles}

\author{Hao Wu}

\thanks{The author was partially supported by NSF grant DMS-1205879.}

\address{Department of Mathematics, The George Washington University, Phillips Hall, Room 708, 801 22nd Stree, NW, Washington DC 20052, USA. Telephone: 1-202-994-0653, Fax: 1-202-994-6760}

\email{haowu@gwu.edu}



\begin{abstract}
We determine the cycle packing number of a directed graph using elementary projective algebraic geometry. Our idea is rooted in the Khovanov-Rozansky theory. In fact, using the Khovanov-Rozansky homology of a graph, we also obtain algebraic methods of detecting directed and undirected cycles containing a particular vertex or edge. 
\end{abstract}

\maketitle

\section{Introduction}\label{sec-intro}

\subsection{Directed paths and cycles in directed graphs} Before stating our results, let us recall some basic concepts and introduce some notations that will be used in this paper.

A \textit{directed graph} is a pair $G=(V(G),E(G))$ of finite sets, where 
\begin{enumerate}
	\item $V(G)$ is the set of vertices of $G$,
	\item $E(G)$ is the set of edges, each of which is directed. That is, of the two vertices at the two ends of each edge, one is the initial vertex, the other is the terminal vertex.
\end{enumerate}
In this paper, we do not assume that graphs are simple. That is, we allow loops (edges initiate and terminate at the same vertex) and multiple edges from one vertex to another. But, to simplify some of our statements, we assume that all graphs in this paper are without vertices of degree $0$.

Given a directed graph $G$, a \textit{directed path} in $G$ from a vertex $u$ to a different vertex $v$ is a sequence $u=v_0,x_0,v_1,x_1,\dots,x_{n-1},v_n=v$ such that
\begin{enumerate}
	\item $v_0,v_1,\dots,v_n$ are pairwise distinct vertices of $G$,
	\item each $x_i$ is an edge of $G$ with initial vertex $v_i$ and terminal vertex $v_{i+1}$.
\end{enumerate}
Two directed paths from $u$ to $v$ are called \textit{edge-disjoint} if they have no common edges. The amount of directed paths in $G$ from $u$ to $v$ is often measured by:
\begin{itemize}
	\item the maximal number $\alpha_{u\rightarrow v}(G)$ of pairwise edge-disjoint directed paths in $G$ from $u$ to $v$.
\end{itemize}
By the Edge Version of Directed Menger's Theorem, 
\begin{eqnarray*}
\alpha_{u\rightarrow v}(G) & = & \text{minimal number of edges in } G \text{ whose removal from } G \\
&& \text{destroys all directed paths in } G \text{ from } u \text{ to }v.
\end{eqnarray*}
In this paper, we will also consider the following naive upper bound for $\alpha_{u\rightarrow v}(G)$.
\begin{itemize}
	\item the minimal number $\beta_{u\rightarrow v}(G)$ of edges in $G$ incident at $u$ or $v$ whose removal from $G$ destroys all directed paths in $G$ from $u$ to $v$.
\end{itemize}

A \textit{directed cycle} in $G$ is a closed directed path, that is, a sequence $v_0,x_0,v_1,x_1,\dots,x_{n-1},v_n,x_n,v_{n+1}=v_0$ satisfying
\begin{enumerate}
	\item $v_0,v_1,\dots,v_n$ are pairwise distinct vertices of $G$,
	\item each $x_i$ is an edge of $G$ with initial vertex $v_i$ and terminal vertex $v_{i+1}$.
\end{enumerate}
Note that two such sequences represent the same directed cycle if one is a circular permutation of the other. That is, the directed cycle given by $v_0,x_0,v_1,x_1,\dots,x_{n-1},v_n,x_n,v_0$ is the same as the one given by $v_1,x_1,\dots,x_{n-1},v_n,x_n,v_0,x_0,v_1$.

We call a directed graph \textit{acyclic} if it does not contain any directed cycles. 

Two directed cycles in $G$ are called \textit{edge-disjoint} if they have no common edges. Two directed cycles in $G$ are called \textit{disjoint} if they have no common vertices.\footnote{Clearly, disjoint directed cycles are also edge-disjoint.}

For a directed graph $G$, we define
\begin{itemize}
	\item $\alpha(G)$ to be the maximal number of pairwise edge-disjoint directed cycles in $G$,
  \item $\tilde{\alpha}(G)$ to be the maximal number of pairwise disjoint directed cycles in $G$,
	\item $\beta(G)$ to be the minimal number of edges in $G$ whose removal from $G$ destroys all directed cycles in $G$.
\end{itemize}
$\alpha(G)$ is known as the \textit{cycle packing number} of $G$. We call $\tilde{\alpha}(G)$ the \textit{strong cycle packing number} of $G$. $\beta(G)$ is known as the \textit{cyclomatic number} of $G$. Clearly, $\tilde{\alpha}(G) \leq \alpha(G) \leq \beta(G)$. Moreover, by the Lucchesi-Younger Theorem \cite{Lucchesi-Younger}, $\alpha(G) = \beta(G)$ if $G$ is planar.

For a vertex $v$ of $G$, we define
\begin{itemize}
	\item $\alpha_v(G)$ to be the maximal number of pairwise edge-disjoint directed cycles in $G$, each of which contains $v$,
	\item $\beta_v(G)$ to be the minimal number of edges incident at $v$ whose removal from $G$ destroys all directed cycles in $G$ containing $v$.
\end{itemize}
Again, $\alpha_v(G) \leq \beta_v(G)$.

More generally, for a subset $E$ of $E(G)$, define 
\begin{itemize}
	\item $\alpha_E(G)$ to be the maximal number of pairwise edge-disjoint directed cycles in $G$, each of which contains at least one edge in $E$.
	\item $\beta_E(G)$ to be the minimal number of edges in $E$ whose removal from $G$ destroys all directed cycles in $G$ containing at least one edge in $E$.
\end{itemize}
Once more, $\alpha_E(G) \leq \beta_E(G)$.

For two distinct vertices $u$ and $v$ of the directed graph $G$, define a directed graph $G_{u \rightarrow v}$ by
\begin{enumerate}
	\item deleting all edges that have $u$ as their terminal vertex,
	\item deleting all edges that have $v$ as their initial vertex,
	\item after the previous two steps, identifying the vertices $u$ and $v$.
\end{enumerate}
Denote by $u\# v$ the vertex of $G_{u \rightarrow v}$ from the vertices $u$ and $v$ of $G$. One can see:
\begin{enumerate}
	\item The natural inclusion map $E(G_{u \rightarrow v}) \hookrightarrow E(G)$ induces a one-to-one correspondence between directed paths in $G$ from $u$ to $v$ and directed cycles in $G_{u \rightarrow v}$ containing $u\# v$.
	\item Under this correspondence, a collection of directed paths in $G$ from $u$ to $v$ is pairwise edge-disjoint if and only if the corresponding collection of directed cycles in $G_{u \rightarrow v}$ containing $u\# v$ is pairwise edge-disjoint. 
	\item $\alpha_{u\rightarrow v}(G) = \alpha_{u\# v}(G_{u \rightarrow v})$ and $\beta_{u\rightarrow v}(G) = \beta_{u\# v}(G_{u \rightarrow v})$.
\end{enumerate}

A \textit{flow network} $N$ is a quadruple $N=(V,c,s,t)$, where 
\begin{enumerate}
	\item $V$ is a finite set,
	\item $s$ and $t$ are distinct elements of $V$, called source and sink of $N$,
	\item $c:V \times V \rightarrow \mathbb{R}_{\geq 0}$ is the capacity function of $N$ satisfying $c(v,v)=c(v,s)=c(t,v)=0$ for all $v \in V$.
\end{enumerate}
A \textit{flow} $f$ of $N$ is a function $f:V \times V \rightarrow \mathbb{R}_{\geq 0}$ such that
\begin{enumerate}
	\item $f(u,v) \leq c(u,v)$ for all $u,~v \in V$,
	\item $\sum_{v \in V} f(u,v) = \sum_{v \in V} f(v,u)$ for all $u \in V \setminus \{s,t\}$. 
\end{enumerate}
The \textit{flow number} of $f$ is $|f|=\sum_{v\in V} f(s,v) = \sum_{v\in V} f(v,t)$. The maximal flow number of $N$ is $|N| := \max \{|f| ~|~ f$ is a flow of $N\}$.

For the flow network $N$, we define an associated directed graph $G_N$ by
\begin{enumerate}
	\item $V(G_N)=V$,
	\item for any pair $(u,v) \in V\times V$, there are exactly $\left\lceil c(u,v) \right\rceil$ directed edges from $u$ to $v$ in $G_N$, where $\left\lceil c \right\rceil$ is the least integer not less than $c$.
\end{enumerate}
One can check that the maximal flow number $|N|$ of $N$ satisfies $|N|\leq \alpha_{s\rightarrow t}(G_N) \leq \beta_{s\rightarrow t}(G_N)$.

\subsection{Incidence ideal and incidence set} Before stating the definition of the incidence ideal and the incidence set, we first recall the definition of elementary symmetric polynomials. For each $l \geq 1$, denote by $e_l(x_1,\dots,x_m)$ the $l$-th elementary symmetric polynomial in the variables $x_1,\dots,x_m$. That is, 
\begin{equation}\label{eq-def-e}
e_l(x_1,\dots,x_m) = \begin{cases}
1 & \text{if } l=0, \\
\sum_{1\leq i_1<i_2<\cdots<i_l\leq m} x_{i_1}x_{i_2}\cdots x_{i_l} & \text{if } 1\leq l\leq m, \\
0 & \text{if } l>m.
\end{cases}
\end{equation}

\begin{definition}\label{def-incidence}
Let $G$ be a directed graph. The polynomial ring $\zed[E(G)]$ is $\zed$-graded so that every $x \in E(G)$ is homogeneous of degree $1$. 

For a vertex $v$ of $G$, suppose $x_1,\dots,x_m$ are the edges having $v$ as their initial vertex, and $y_1,\dots,y_n$ are the edges having $v$ as their terminal vertex.\footnote{It is possible to have $x_i=y_j$ since we allow loops.} Set $k_v = \max\{m,n\}$. For $1\leq l \leq k_v$, define 
\begin{equation}\label{eq-incidence}
\delta_{v,l}:= e_l(x_1,\dots,x_m) - e_l(y_1,\dots,y_n) \in \zed[E(G)] 
\end{equation}
and call it the degree $l$ incidence relation at $v$. 

Let $\Delta_G := \{\delta_{v,l}~|~ v \in V(G), ~1\leq l \leq k_v\}$ be the set (counting multiplicity) of all incidence relations in $G$. Denote by $I(G)$ the ideal of $\zed[E(G)]$ generated by the set $\Delta_G$. We call $I(G)$ the incidence ideal for $G$. Note that $I(G)$ is a homogeneous ideal. Therefore, it defines a complex projective algebraic set 
\[
P(G) := \{p\in \mathbb{CP}^{|E(G)|-1} ~|~ f(p)=0, ~\forall ~f\in I(G)\}.
\]
We call $P(G)$ the incidence set of $G$.
\end{definition}

Before stating our main result, we introduce the cycle spectrum of a directed graph.

\begin{definition}\label{def-cycle-spectrum}
Let $G$ be a directed graph. A collection $\mathcal{C}$ of pairwise edge-disjoint directed cycles in $G$ is called maximal if it is not a subcollection of any collection of pairwise edge-disjoint directed cycles in $G$. In other words, $\mathcal{C}$ is maximal if, after removing all edges belonging to directed cycles in $\mathcal{C}$ from $G$, the remaining directed graph is acyclic. 

Define $\gamma_n(G)$ to be the number of maximal collections of pairwise edge-disjoint directed cycles in $G$ containing exactly $n$ cycles. We call the sequence $\Gamma(G):= \{\gamma_n(G)\}_{n=1}^\infty$ the cycle spectrum of $G$. Note that any collection of $\alpha(G)$ pairwise edge-disjoint directed cycles in $G$ is maximal. But it is possible to have maximal collections of pairwise edge-disjoint directed cycles containing less than $\alpha(G)$ cycles. Also, the cycle spectrum $\Gamma(G)$ contains only finitely many non-zero terms since $\gamma_n(G)=0$ if $n > \alpha(G)$.
\end{definition}

Our main result is that the incidence set determines the cycle packing number and the cycle spectrum.

\begin{theorem}\label{thm-incidence}
Let $G$ be any directed graph. Then:
\begin{enumerate}
	\item $P(G)$ is a union of finitely many linear subspaces of $\mathbb{CP}^{|E(G)|-1}$;
	\item $\dim P(G) = \alpha(G)-1$;
	\item $\deg P(G)=\gamma_{\alpha(G)}(G)$;
	\item $\gamma_n(G)$ is equal to the number of irreducible components of $P(G)$ of dimension $n-1$.
\end{enumerate}
\end{theorem}

Slightly modifying the definition of the incidence set, we get a method of determining the strong cycle packing number and the strong cycle spectrum defined below.

\begin{definition}\label{def-incidence-strong}
Let $G$ be a directed graph. For a vertex $v$ of $G$, suppose $x_1,\dots,x_m$ are the edges having $v$ as their initial vertex, and $y_1,\dots,y_n$ are the edges having $v$ as their terminal vertex. Recall that $k_v = \max\{m,n\}$. The set of strong incidence relations of $G$ at $v$ is 
\[
\tilde{\Delta}_v:=\{e_1(x_1,\dots,x_m) - e_1(y_1,\dots,y_n)\}\cup \{e_l(x_1,\dots,x_m) ~|~ 2 \leq l \leq m\} \cup \{e_l(y_1,\dots,y_n) ~|~ 2 \leq l \leq n\}.
\]

Denote by $\tilde{I}(G)$ the ideal of $\zed[E(G)]$ generated by the set $\tilde{\Delta}_G:= \bigcup_{v\in V(G)}\tilde{\Delta}_v$. We call $\tilde{I}(G)$ the strong incidence ideal for $G$. Note that $\tilde{I}(G)$ is a homogeneous ideal. Therefore, it defines a complex projective algebraic set 
\[
\tilde{P}(G) := \{p\in \mathbb{CP}^{|E(G)|-1} ~|~ f(p)=0, ~\forall ~f\in \tilde{I}(G)\}.
\]
We call $\tilde{P}(G)$ the strong incidence set of $G$.

A collection $\mathcal{C}$ of pairwise disjoint directed cycles in $G$ is called strongly maximal if it is not a subcollection of any collection of pairwise disjoint directed cycles in $G$. In other words, $\mathcal{C}$ is strongly maximal if, after removing all edges incident at vertices of cycles in $\mathcal{C}$ from $G$, the remaining directed graph is acyclic. 

Define $\tilde{\gamma}_n(G)$ to be the number of strongly maximal collections of pairwise disjoint directed cycles in $G$ containing exactly $n$ cycles. We call the sequence $\tilde{\Gamma}(G):= \{\tilde{\gamma}_n(G)\}_{n=1}^\infty$ the strong cycle spectrum of $G$.
\end{definition}

Clearly, we have that $I(G) \subset \tilde{I}(G)$, $P(G) \supset \tilde{P}(G)$ and $\tilde{\gamma}_n(G) \leq \gamma_n(G)$.

\begin{corollary}\label{cor-incidence-strong}
Let $G$ be any directed graph. Then:
\begin{enumerate}
	\item $\tilde{P}(G)$ is a union of finitely many linear subspaces of $\mathbb{CP}^{|E(G)|-1}$;
	\item $\dim \tilde{P}(G) = \tilde{\alpha}(G)-1$;
	\item $\deg \tilde{P}(G)=\tilde{\gamma}_{\tilde{\alpha}(G)}(G)$;
	\item $\tilde{\gamma}_n(G)$ is equal to the number of irreducible components of $\tilde{P}(G)$ of dimension $n-1$.
\end{enumerate}
\end{corollary}

There are several general purpose software packages for algebraic geometry, such as CoCoA, Macaulay2 and Singular. One can use these packages to compute dimensions and degrees. So, for small directed graphs, Theorem \ref{thm-incidence} and Corollary \ref{cor-incidence-strong} provide an automated method of computing $\alpha(G)$, $\gamma_{\alpha(G)}(G)$ and $\tilde{\alpha}(G)$, $\tilde{\gamma}_{\tilde{\alpha}(G)}(G)$.

A natural question is that, for what $G$, is $P(G)$ or $\tilde{P}(G)$ a projective variety, that is, an irreducible projective algebraic set? Knowing Theorem \ref{thm-incidence}, it is relatively easy to find a combinatorial answer to this question. 

\begin{theorem}\label{thm-Krull-sharp-varieties}
Let $G$ be a directed graph. 
\begin{enumerate}
\item The following statements are equivalent:
\begin{enumerate}
	\item $P(G)$ is a projective variety;
	\item $P(G)$ is a linear subspace of $\mathbb{CP}^{|E(G)|-1}$;
	\item $G$ contains exactly $\alpha(G)$ distinct directed cycles.
\end{enumerate}
\item The following statements are equivalent:
\begin{enumerate}
	\item $\tilde{P}(G)$ is a projective variety;
	\item $\tilde{P}(G)$ is a linear subspace of $\mathbb{CP}^{|E(G)|-1}$;
	\item $G$ contains exactly $\tilde{\alpha}(G)$ distinct directed cycles.
\end{enumerate}
\item If $\tilde{P}(G)$ is a projective variety, then $P(G)=\tilde{P}(G)$.
\end{enumerate}

\end{theorem}

We will prove Theorem \ref{thm-incidence}, Corollary \ref{cor-incidence-strong} and Theorem \ref{thm-Krull-sharp-varieties} in Section \ref{sec-Krull-sharp-varieties} below.

\subsection{Khovanov-Rozansky homology and its Krull dimension} In \cite{KR1,KR2}, Khovanov and Rozansky introduced an innovative method of constructing link homologies whose graded Euler characteristics are versions of the HOMFLYPT polynomial. Such constructions are known as categorifications in knot theory. The homologies constructed by Khovanov and Rozansky are now called the Khovanov-Rozansky homology. Since their initial work, the Khovanov-Rozansky homology has been generalized and re-interpreted by many researchers. See, for example, \cite{Mackaay-Stosic-Vaz2,Ras2,Webster-Williamson,Wu-color} and many more. Khovanov and Rozansky's construction is a two step process. First, they defined a Koszul matrix factorization for each MOY graph.\footnote{A MOY graph is a directed graph such that the in-degree of each vertex is equal to the out-degree of the same vertex. The name is an acronym of the authors of \cite{MOY}.} Then, using the crossing information in a link diagram, they constructed a chain complex of matrix factorizations for each link diagram. The hard part of their work is to prove that their homology is independent of the choice of the diagram of a link.

The Khovanov-Rozansky homology of directed graphs defined below is a straightforward generalization of the first step in the construction of the HOMFLYPT version of the Khovanov-Rozansky homology of MOY graphs given in \cite{Ras2}. The only change is that we no longer require the graph to be MOY. 

\begin{definition}\label{def-KR}
Let $G$ be a directed graph. Recall that $\Delta_G$ is the set (counting multiplicity) of all incidence relations in $G$. 

The Khovanov-Rozansky chain complex $\mathscr{C}_\ast(G)$ is the graded Koszul chain complex $C_\ast^{\zed[E(G)]} (\Delta_G)$ over $\zed[E(G)]$ defined by $\Delta_G$.\footnote{Strictly speaking, we need to specify a linear order for $\Delta_G$. But, as we shall see in Section \ref{sec-KR}, a change of such a linear order does not change the isomorphism type of the Khovanov-Rozansky chain complex.} The Khovanov-Rozansky homology $\mathscr{H}_\ast (G)$ of $G$ is the homology of $\mathscr{C}_\ast(G)$.

Note that, as a graded Koszul chain complex, $\mathscr{C}_\ast (G)$ has a $\zed \oplus \zed$-grading. One is the homological grading, and the other is the $\zed$-grading of the underlying $\zed[E(G)]$-module. Clearly, $\mathscr{H}_\ast (G)$ inherits this $\zed \oplus \zed$-grading.

There is also the notion of the Khovanov-Rozansky homology over $\Q$. In this case, the chain complex is $\mathscr{C}_\ast^\Q(G) := \Q\otimes_{\zed}\mathscr{C}_\ast(G)$ and the homology is the homology $\mathscr{H}_\ast^\Q(G)$ of $\mathscr{C}_\ast^\Q(G)$.
\end{definition}

\begin{lemma}\label{lemma-KR-invariance}
If directed graphs $G$ and $G'$ are isomorphic, then $\mathscr{C}_\ast(G) \cong \mathscr{C}_\ast(G')$ as chain complexes of graded $\zed[E(G)]$-modules. Therefore, $\mathscr{H}_\ast (G) \cong \mathscr{H}_\ast (G')$ as $\zed\oplus\zed$-graded $\zed[E(G)]$-modules.
\end{lemma}

\begin{proof}
A graph isomorphism gives one-to-one correspondences $V(G) \rightarrow V(G')$ and $E(G) \rightarrow E(G')$ that, when combined, preserve the incidence relations.
\end{proof}

Next we recall the definition of the Krull dimension. Properties of the Krull dimension relevant to our proofs will be reviewed in Section \ref{sec-Krull-cycles}.

\begin{definition}\label{def-Krull}
Let $R$ be a commutative ring with $1$. The Krull dimension of $R$ is 
\[
\Kdim R:= \sup\{n\geq 0 ~| \text{ there are } n+1 \text{ prime ideals } \mathfrak{p}_0, \mathfrak{p}_1,\dots,\mathfrak{p}_n \text{ of } R \text{ such that } \mathfrak{p}_0 \subsetneq \mathfrak{p}_1\subsetneq \cdots \subsetneq\mathfrak{p}_n \subsetneq R\}.
\]
For an $R$-module $M$, define $\ann_R(M)=\{r\in R~|~ rm=0 ~\forall m\in M\}$, which is an ideal of $R$. Then 
\[
\Kdim_R M := \Kdim (R/\ann_R(M)).
\]
\end{definition}

It is possible to have $\Kdim R = +\infty$. But this does not happen for the rings we will discuss in this paper. Also, for an ideal $I$ of $R$, $\ann_R(R/I)=I$. So $\Kdim_R(R/I) = \Kdim (R/I)$, where the left hand side is the Krull dimension of an $R$-module, and the right hand side is the Krull dimension of a ring.

The following fact can be found in many introductory books to projective algebraic geometry.

\begin{fact}\label{fact-kdim}
For a homogeneous ideal $I$ of $\Q[x_0,\dots,x_n]$, denote by $X(I)$ the projective algebraic set 
\[
X(I)= \{(x_0:\cdots:x_n) \in \mathbb{CP}^n ~|~ f(x_0,\dots,x_n)=0 ~(\forall ~f\in I)\}.
\]
Then $\dim X(I) = \Kdim \Q[x_0,\dots,x_n]/I-1$.
\end{fact}

By Lemma \ref{lemma-KR-module} below, for a directed graph $G$, its $0$-th Khovanov-Rozansky homology is the graded $\zed[E(G)]$-module $\mathscr{H}_0 (G) \cong \zed[E(G)] / I(G)$, and $\Kdim_{\zed[E]}\mathscr{H}_\ast (G) =\Kdim_{\zed[E]}\mathscr{H}_0 (G)$ for any $E \subset E(G)$. So, in the discussions below, we state results in terms of Krull dimensions of $\mathscr{H}_0 (G)$ instead of the whole Khovanov-Rozansky homology $\mathscr{H}_\ast (G)$.

\subsection{Khovanov-Rozansky homology and directed cycles}\label{subsec-main-results}
 
\begin{theorem}\label{thm-KR-detect}
Let $G$ be a directed graph.
\begin{enumerate}[1.]
	\item The following statements are equivalent:
\begin{enumerate}[(1)]
	\item $G$ is acyclic;
	\item $\mathscr{H}_0 (G)$ is a finitely generated $\zed$-module;
	\item $\mathscr{H}_\ast(G)$ is a finitely generated $\zed$-module;
	\item There is an $n \geq 1$ such that $x^n \in I(G)$ for all $x \in E(G)$, where $I(G)$ is the incidence ideal of $\zed[E(G)]$;
	\item $\Kdim_{\zed[E(G)]} \mathscr{H}_0 (G) =1$;
	\item $\Kdim_{\Q[E(G)]} \mathscr{H}_0^\Q (G) =0$.
\end{enumerate}
  \item $\alpha(G) = \Kdim_{\Q[E(G)]} \mathscr{H}_0^\Q (G) \leq \Kdim_{\zed[E(G)]} \mathscr{H}_0 (G) -1 \leq \beta(G)$. In particular, if $G$ is planar, then $\alpha(G) = \Kdim_{\Q[E(G)]} \mathscr{H}_0^\Q (G) = \Kdim_{\zed[E(G)]} \mathscr{H}_0 (G) -1 = \beta(G)$.
\end{enumerate}
\end{theorem}

As a byproduct of the proof of Theorem \ref{thm-KR-detect}, we have the following general proposition.

\begin{proposition}\label{prop-Krull-bound-E}
For a directed graph $G$ and $E \subset E(G)$, 
\[
\alpha_E(G) \leq \Kdim_{\Q[E]} \mathscr{H}_0^\Q (G) \leq \Kdim_{\zed[E]} \mathscr{H}_0 (G) -1.
\]
\end{proposition}

If $E$ is the subset of all edges incident at a vertex $v$ of $G$, then we have stronger results.

\begin{theorem}\label{thm-KR-detect-vertex} 
Let $G$ be a directed graph, and $v$ a vertex in $G$. Denote by $E(v)$ the set of all edges of $G$ incident at $v$.
\begin{enumerate}[1.]
	\item The following statements are equivalent:
\begin{enumerate}[(1)]
	\item There are no directed cycles in $G$ containing $v$;
	\item $\zed[E(v)]/\ann_{\zed[E(v)]}(\mathscr{H}_0 (G))$ is a finitely generated $\zed$-module;
	\item $\zed[E(v)]/\ann_{\zed[E(v)]}(\mathscr{H}_\ast (G))$ is a finitely generated $\zed$-module;
	\item There is an $n \geq 1$ such that $x^n \in I(G)$ for all $x \in E(v)$, where $I(G)$ is the incidence ideal of $\zed[E(G)]$;
	\item $\Kdim_{\zed[E(v)]} \mathscr{H}_0 (G) =1$;
	\item $\Kdim_{\Q[E(v)]} \mathscr{H}_0^\Q (G) =0$.
\end{enumerate}
  \item $\alpha_v(G) \leq \Kdim_{\Q[E(v)]} \mathscr{H}_0^\Q (G) \leq \Kdim_{\zed[E(v)]} \mathscr{H}_0 (G) -1 \leq \beta_v(G)$.
\end{enumerate}
\end{theorem}

One can also detect directed cycles containing a particular edge using $\mathscr{H}_0$.

\begin{corollary}\label{cor-KR-detect-edge} 
Let $G$ be a directed graph, and $x$ an edge in $G$. The following statements are equivalent:
\begin{enumerate}[(1)]
	\item There are no directed cycles in $G$ containing $x$;
	\item $\zed[x]/\ann_{\zed[x]}(\mathscr{H}_0 (G))$ is a finitely generated $\zed$-module;
	\item $\zed[x]/\ann_{\zed[x]}(\mathscr{H}_\ast (G))$ is a finitely generated $\zed$-module;
	\item There is an $n \geq 1$ such that $x^n \in I(G)$, where $I(G)$ is the incidence ideal of $\zed[E(G)]$;
	\item $\Kdim_{\zed[x]} \mathscr{H}_0 (G) =1$;
	\item $\Kdim_{\Q[x]} \mathscr{H}_0^\Q (G) =0$.
\end{enumerate}
There is a directed cycle in $G$ containing $x$ if and only if $\Kdim_{\Q[x]} \mathscr{H}_0^\Q (G) = \Kdim_{\zed[x]} \mathscr{H}_0 (G) -1 =1$.
\end{corollary}

Corollary \ref{cor-KR-detect-edge} follows easily from Theorem \ref{thm-KR-detect-vertex}. We will prove Theorems \ref{thm-KR-detect}, \ref{thm-KR-detect-vertex}, Proposition \ref{prop-Krull-bound-E} and Corollary \ref{cor-KR-detect-edge} in Section \ref{sec-Krull-cycles}.

Using the aforementioned relation of directed paths and cycles, we have the following immediate corollaries.

\begin{corollary}\label{cor-KR-detect-path} 
Let $G$ be a directed graph, and $u$, $v$ two distinct vertices in $G$. Denote by $E(u\# v)$ the set of all edges of $G_{u\rightarrow v}$ incident at $u\# v$.
\begin{enumerate}[1.]
	\item The following statements are equivalent:
\begin{enumerate}[(1)]
	\item There are no directed paths in $G$ from $u$ to $v$;
	\item $\zed[E(u\# v)]/\ann_{\zed[E(u\# v)]}(\mathscr{H}_0 (G_{u\rightarrow v}))$ is a finitely generated $\zed$-module;
	\item $\zed[E(u\# v)]/\ann_{\zed[E(u\# v)]}(\mathscr{H}_\ast (G_{u\rightarrow v}))$ is a finitely generated $\zed$-module;
	\item $\Kdim_{\zed[E(u\# v)]} \mathscr{H}_0 (G_{u\rightarrow v}) =1$;
	\item $\Kdim_{\Q[E(u\# v)]} \mathscr{H}_0^\Q (G_{u\rightarrow v}) =0$.
\end{enumerate}
  \item $\alpha_{u\rightarrow v} (G) \leq \Kdim_{\Q[E(u\# v)]} \mathscr{H}_0^\Q (G_{u\rightarrow v}) \leq \Kdim_{\zed[E(u\# v)]} \mathscr{H}_0 (G_{u\rightarrow v}) -1 \leq \beta_{u\rightarrow v}(G)$.
\end{enumerate}
\end{corollary}

\begin{proof}
Follows immediately from Theorem \ref{thm-KR-detect-vertex}.
\end{proof}

\begin{corollary}\label{cor-flow}
Let $N$ be a flow network of with source $s$ and sink $t$, and $G_N$ the directed graph associated to $N$. Denote by $(G_N)_{s \rightarrow t}$ the directed graph obtained from $G_N$ by identifying $s$, $t$ and by $s \# t$ the vertex in $(G_N)_{s \rightarrow t}$ from $s$, $t$. Then the maximal flow number $|N|$ of $N$ satisfies $|N|\leq \alpha_{s \rightarrow t}(G_N) \leq \Kdim_{\Q[E(s\# t)]} \mathscr{H}_0^\Q ((G_N)_{s \rightarrow t}) \leq \Kdim_{\zed[E(s\# t)]} \mathscr{H}_0 ((G_N)_{s \rightarrow t}) -1 \leq \beta_{{s \rightarrow t}}(G_N)$, where $E(s\# t)$ is the set of edges in $(G_N)_{s \rightarrow t}$ incident at $s \# t$.
\end{corollary}

\begin{proof}
Follows immediately from Part 2 of Corollary \ref{cor-KR-detect-path}.
\end{proof}

\subsection{Undirected cycles} The Khovanov-Rozansky homology also detects undirected cycles in a directed graph. Here, an undirected cycle in a directed graph $G$ is a sequence 
\[
v_0,x_0,v_1,x_1,\dots,x_{n-1},v_n,x_n,v_{n+1}=v_0
\] 
satisfying
\begin{enumerate}
	\item $v_0,v_1,\dots,v_n$ are pairwise distinct vertices of $G$,
	\item each $x_i$ is an edge of $G$ either with initial vertex $v_i$ and terminal vertex $v_{i+1}$ or with initial vertex $v_{i+1}$ and terminal vertex $v_{i}$
\end{enumerate}
Again, two such sequences represent the same undirected cycle if one is a circular permutation of the other. That is, the undirected cycle given by $v_0,x_0,v_1,x_1,\dots,x_{n-1},v_n,x_n,v_0$ is the same as the one given by $v_1,x_1,\dots,x_{n-1},v_n,x_n,v_0,x_0,v_1$.

We formulate our results on undirected cycles as Theorems \ref{thm-tree} and \ref{thm-tree-v}.

\begin{theorem}\label{thm-tree}
Let $G$ be a directed graph. 
\begin{enumerate}[1.]
	\item The underlying undirected graph of $G$ is a disjoint union of trees if and only if $\mathscr{H}_\ast(G) \cong \mathscr{H}_0(G) \cong \zed$ as graded $\zed[E(G)]$-modules, where $\zed$ is the graded $\zed[E(G)]$-module $\zed = \zed[E(G)]/(E(G))$, and $(E(G))$ is the homogeneous ideal of $\zed[E(G)]$ generated by $E(G)$.
	\item $G$ contains an undirected cycle if and only if $\mathscr{H}_{0,1} (G) \neq 0$, where $\mathscr{H}_{0,1} (G)$ is the homogeneous component of $\mathscr{H}_{0} (G)$ of module degree $1$.
	\item Define 
	\begin{itemize}
		\item $\alpha_{undirected}(G)$ to be the maximal number of pairwise edge-disjoint undirected cycles in $G$,
		\item $\beta_{undirected}(G)$ to be the minimal number of edges in $G$ whose removal from $G$ destroys all undirected cycles in $G$.
	\end{itemize}
Then $\alpha_{undirected}(G) \leq \rank \mathscr{H}_{0,1} (G) \leq \beta_{undirected}(G)$, where $\rank \mathscr{H}_{0,1} (G)$ is the rank of $\mathscr{H}_{0,1} (G)$ as a $\zed$-module.
\end{enumerate}
\end{theorem}

\begin{theorem}\label{thm-tree-v}
Let $G$ be a directed graph and $v$ a vertex of $G$. Denote by $E(v)$ the set of edges of $G$ incident at $v$. Define $A(v) = \zed[E(v)]/(\ann_{\zed[E(v)]} \mathscr{H}_0 (G))$, which is a graded $\zed[E(v)]$-module.
\begin{enumerate}[1.]
		\item There are no undirected cycles in $G$ containing $v$ if and only if $A(v) \cong \zed$ as graded $\zed[E(v)]$-modules, where $\zed$ is the graded $\zed[E(v)]$-module $\zed = \zed[E(v)]/(E(v))$, and $(E(v))$ is the homogeneous ideal of $\zed[E(v)]$ generated by $E(v)$.
	\item There is an undirected cycle in $G$ containing $v$ if and only if $A_1(v) \neq 0$, where $A_1(v)$ is the homogeneous component of $A(v)$ of module degree $1$.
	\item Define 
	\begin{itemize}
		\item $\alpha_{undirected}(G,v)$ to be the maximal number of pairwise edge-disjoint undirected cycles in $G$, each of which contains $v$,
		\item $\beta_{undirected}(G,v)$ to be the minimal number of edges in $G$ incident at $v$ whose removal from $G$ destroys all undirected cycles in $G$ containing $v$.
	\end{itemize}
Then $\alpha_{undirected}(G,v) \leq \rank A_1(v) \leq \beta_{undirected}(G,v)$.
\end{enumerate}
\end{theorem}

The proofs of Theorems \ref{thm-tree} and \ref{thm-tree-v} have the same flavor as that of Theorems \ref{thm-KR-detect} and \ref{thm-KR-detect-vertex}, but are more elementary. We include these in Section \ref{sec-cycles-u}.

It turns out that $\mathscr{H}_{0,1} (G)$ is closely related to the vertex-edge incidence matrix. For a directed graph $G$, we order its vertices and edges. That is, we write $V(G)=\{v_1,\dots,v_m\}$ and $E(G)=\{x_1,\dots,x_n\}$. Denote by $\zed \cdot E(G)$ the free $\zed$-module generated by $E(G)$. Identify the $\zed$-modules $\zed \cdot E(G)$ and $\zed^n$ by the isomorphism which identifies each $x_j$ with the row vector with a single $1$ at the $j$-th position and $0$'s at all other positions. The vertex-edge incidence matrix $D=(d_{i,j})_{m\times n}$ of $G$ is defined by 
\[
d_{i,j}=
\begin{cases}
1  & \text{if } v_i \text{ is the terminal vertex, but not the initial vertex of } x_j,\\
-1 & \text{if } v_i \text{ is the initial vertex, but not the terminal vertex of } x_j,\\
0  & \text{otherwise.} 
\end{cases}
\]
Clearly, the above identification maps the degree $1$ incidence relation $\delta_{v,1}$ at a vertex $v$ to the row in $D$ corresponding to $v$. Note that $\mathscr{H}_{0,1} (G)$ is isomorphic to $\zed \cdot E(G)$ modulo all degree $1$ incidence relations in $G$, which is isomorphic to $\zed^n$ modulo the submodule generated by rows of the vertex-edge incidence matrix. In view of this, Theorem \ref{thm-tree} implies that the following statements are equivalent:
\begin{enumerate}
	\item There is an undirected cycle in $G$.
	\item The submodule of $\zed^n$ generated by rows of $D$ is not $\zed^n$.
	\item The submodule of $\zed^n$ generated by rows of $D$ has rank less than $n$.
\end{enumerate}
This and some other related results were known prior to our work. 

The following corollary is a re-formulation of Theorems \ref{thm-tree} and \ref{thm-tree-v} using linear algebra over $\zed_2$ without referring to the Khovanov-Rozansky homology. Versions of this corollary were known prior to our work.

\begin{corollary}\label{cor-u-cycle-detect}
Let $G$ be an undirected graph with finite vertex set $V(G)$ and finite edge set $E(G)$. For any $v \in V(G)$, denote by $E(v)$ the set of edges in $G$ incident at $v$ and by $L(v)$ the set of loops in $G$ at $v$ (edges connecting $v$ to $v$). Define $\zed_2 \cdot E(G) := \bigoplus_{x \in E(G)} \zed_2 \cdot x$, which is a linear space over the field $\zed_2$. Define $S$ to be the subspace of $\zed_2 \cdot E(G)$ spanned by $\{\sum_{x\in E(v)\setminus L(v)} x ~|~ v \in V(G)\}$.
\begin{enumerate}[1.]
	\item $\alpha_{undirected}(G) \leq |E(G)|-\dim_{\zed_2} S \leq \beta_{undirected}(G)$.\footnote{Strictly speaking, $\alpha_{undirected}(G)$, $\beta_{undirected}(G)$ and $\alpha_{undirected}(G,v)$, $\beta_{undirected}(G,v)$ are defined above for directed graphs. But these are clearly independent of the directions of edges in $G$ and are therefore invariants for undirected graphs.} In particular, $|E(G)|>\dim_{\zed_2} S$ if and only if there is an undirected cycle in $G$.
	\item Define $\zed_2 \cdot E(v) := \bigoplus_{x \in E(v)} \zed_2 \cdot x$, which is a subspace of $\zed_2 \cdot E(G)$. Then 
	\[
	\alpha_{undirected}(G,v) \leq |E(v)|-\dim_{\zed_2} (\zed_2 \cdot E(v))\cap S = \dim_{\zed_2} (\zed_2 \cdot E(v)+ S) - \dim_{\zed_2} S \leq \beta_{undirected}(G,v).
	\]
	In particular, $|E(v)|> \dim_{\zed_2}  (\zed_2 \cdot E(v))\cap S$ if and only if there is an undirected cycle in $G$ containing $v$.
	\item For an edge $x$ of $G$, there exists an undirected cycle in $G$ containing $x$ if and only if $x \notin S$.  
\end{enumerate}
\end{corollary}

Parts 1 and 2 of Corollary \ref{cor-u-cycle-detect} follow from Theorems \ref{thm-tree} and \ref{thm-tree-v} via an easy computation of $\zed_2\otimes_{\zed} \mathscr{H}_{0,1} (G)$ and $\zed_2\otimes_{\zed} A_1(v)$. Part 3 follows from a simple direct argument. We include the proof of Corollary \ref{cor-u-cycle-detect} in Section \ref{sec-cycles-u}. 

Part 3 of Corollary \ref{cor-u-cycle-detect} explains how to use the vertex-edge incidence matrix $D$ over $\zed_2$ to detect the existence of undirected cycles in $G$ containing a particular vertex or edge. Let $D'$ be the reduced row echelon matrix over $\zed_2$ obtained by row operations from $D$. Then the non-zero rows of $D'$ form a basis for $S$. By Part 3 of Corollary \ref{cor-u-cycle-detect}, an edge $x_j$ of $G$ is not contained in any undirected cycles in $G$ if and only if $x_j$, viewed as a row vector in $\zed_2^n$, is equal to a row of $D'$. Therefore, a vertex $v$ of $G$ is not contained in any undirected cycles in $G$ if and only if all edges incident at $v$ are equal to rows of $D'$.

\subsection{A homology for undirected graphs} We should mention that one can modify the definition of the Khovanov-Rozansky homology to define a version of it for undirected graphs.

\begin{definition}\label{def-KR-undirected}
Let $G$ be an undirected graph with finite vertex set $V(G)$ and finite edge set $E(G)$. For any vertex $v$ of $G$, denote by $E(v)$ the set of edges in $G$ incident at $v$ counting multiplicities.\footnote{That is, each loop edge at $v$ appears twice in $E(v)$, each non-loop edges incident at $v$ appears once in $E(v)$. This ensures that $|E(v)|=\deg v$.} Define $U_\ast(G)$ to be the graded Koszul chain complex $C_\ast^{\zed[(G)]}(\Omega_G)$ over $\zed[E(G)]$ defined by the sequence $\Omega_G = \{e_l(E(v))~|~ v \in V(G),~1\leq l \leq \deg v\}$, where $e_l(E(v))$ is the elementary symmetric polynomial on $E(v)$. That is, if $E(v)=\{x_1,\dots,x_m\},$\footnote{It is possible to have $x_i=x_j$ since we allow loops and count multiplicities in $E(v)$ in the undirected situation.} then $e_l(E(v))$ is given by Equation \eqref{eq-def-e}.

Define $\mathscr{U}_\ast(G)$ to the homology of $U_\ast(G)$.

Note that, as a graded Koszul chain complex, $U_\ast (G)$ has a $\zed \oplus \zed$-grading. One is the homological grading, and the other is the $\zed$-grading of the underlying $\zed[E(G)]$-module. Clearly, $\mathscr{U}_\ast (G)$ inherits this $\zed \oplus \zed$-grading.
\end{definition}

The following proposition describes some basic properties of $\mathscr{U}_\ast (G)$.

\begin{proposition}\label{prop-KR-undirected}
Let $G$ be an undirected graph. 
\begin{enumerate}[1.]
  \item $\mathscr{U}_\ast(G)$ is a finitely generated $\zed$-module.
	\item $G$ is a disjoint union of trees if and only if $\mathscr{U}_\ast(G) \cong \mathscr{U}_0(G) \cong \zed$ as graded $\zed[E(G)]$-modules, where $\zed$ is the graded $\zed[E(G)]$-module $\zed = \zed[E(G)]/(E(G))$, and $(E(G))$ is the homogeneous ideal of $\zed[E(G)]$ generated by $E(G)$.
	\item $G$ contains an undirected cycle if and only if $\mathscr{U}_{0,1} (G) \neq 0$, where $\mathscr{U}_{0,1} (G)$ is the homogeneous component of $\mathscr{U}_{0} (G)$ of module degree $1$.
	\item $\alpha_{undirected}(G) \leq \dim_{\zed_2}\zed_2\otimes_{\zed} \mathscr{U}_{0,1} (G) \leq \beta_{undirected}(G)$.
\end{enumerate}
\end{proposition}

The proof of Proposition \ref{prop-KR-undirected} is very similar to the proofs of the corresponding results about the Khovanov-Rozansky homology, especially that of  Theorem \ref{thm-tree}. In Section \ref{sec-cycles-u}, we give a sketch of its proof and leave the details to the reader.

\subsection{Remarks and questions} For Proposition \ref{prop-Krull-bound-E}, the author was only able to prove that $\Kdim_{\Q[E]} \mathscr{H}_0^\Q (G) \leq \Kdim_{\zed[E]} \mathscr{H}_0 (G) -1$. But, if this inequality is strict for some graphs, it then means that the Khovanov-Rozansky homology over $\Q$, which should retain less information about the graph than the Khovanov-Rozansky homology over $\zed$, actually provides a stronger upper bound for $\alpha_E(G)$ than the homology over $\zed$. This would certainly be very odd.

\begin{conjecture}\label{conj-Q-Z-same}
For any directed graph $G$ and any subset $E$ of $E(G)$, 
\[
\Kdim_{\Q[E]} \mathscr{H}_0^\Q (G) = \Kdim_{\zed[E]} \mathscr{H}_0 (G) -1.
\]
\end{conjecture}

From Theorems \ref{thm-KR-detect} and \ref{thm-KR-detect-vertex}, we know that, if the subset $E$ of $E(G)$ is $E=E(G)$ or $E(v)$ for some $v \in V(G)$, then $\Kdim_{\zed[E]} \mathscr{H}_0 (G) -1 \leq \beta_E(G)$.

\begin{question}
Is it true that $\Kdim_{\zed[E]} \mathscr{H}_0 (G) -1 \leq \beta_E(G)$ for all $E \subset E(G)$?
\end{question}

In this paper, we only used the $0$-th Khovanov-Rozansky homology. A more comprehensive study of $\mathscr{H}_\ast$ using techniques from homological algebra and commutative algebra should reveal more combinatorial information of the graph implicit in the Khovanov-Rozansky homology.

\begin{question}
What other combinatorial information of a graph is implicit in its Khovanov-Rozansky homology?
\end{question}

\subsection{Organization of this paper} First, in Section \ref{sec-Krull-sharp-varieties}, we prove Theorem \ref{thm-incidence}, Corollary \ref{cor-incidence-strong} and Theorem \ref{thm-Krull-sharp-varieties}. Then, in Section \ref{sec-KR}, we review the definition and properties of graded Koszul chain complexes and deduce some properties of the Khovanov-Rozansky homology. After that, in Section \ref{sec-Krull-cycles}, we recall basic properties of the Krull dimension and prove Theorems \ref{thm-KR-detect} and \ref{thm-KR-detect-vertex}. Finally, in Section \ref{sec-cycles-u}, we discuss undirected cycles. 

\begin{acknowledgments}
This project was initiated while the author was visiting the Math Department of the University of Maryland during his sabbatical. He would like to thank the Department and especially his host, Professor Xuhua He, for their hospitality.
\end{acknowledgments}

\section{Structure of the Incidence Set}\label{sec-Krull-sharp-varieties}
In this section, we study the structure of $P(G)$ and prove Theorems \ref{thm-incidence} and \ref{thm-Krull-sharp-varieties}. Facts about projective algebraic geometry used in this section can be found in \cite{Arrondo-notes}.

\subsection{Disassembling a directed graph} Before stating our definition of disassemblies of a graph, let us first recall the concepts of directed trails and directed circuits.

Given a directed graph $G$, a directed trail in $G$ from a vertex $u$ to a different vertex $v$ is a sequence $u=v_0,x_0,v_1,x_1,\dots,x_{n-1},v_n=v$ such that
\begin{enumerate}
	\item $x_0,x_1,\dots,x_{n-1}$ are pairwise distinct edges of $G$,
	\item each $x_i$ is an edge of $G$ with initial vertex $v_i$ and terminal vertex $v_{i+1}$.
\end{enumerate}
A directed circuit in $G$ is a closed trial, that is, a sequence $v_0,x_0,v_1,x_1,\dots,x_{n-1},v_n,x_n,v_{n+1}=v_0$ satisfying
\begin{enumerate}
	\item $x_0,x_1,\dots,x_n$ are pairwise distinct edges of $G$,
	\item each $x_i$ is an edge of $G$ with initial vertex $v_i$ and terminal vertex $v_{i+1}$.
\end{enumerate}
Two such sequences represent the same directed circuit if one is a circular permutation of the other. That is, the directed circuit given by $v_0,x_0,v_1,x_1,\dots,x_{n-1},v_n,x_n,v_0$ is the same as the one given by $v_1,x_1,\dots,x_{n-1},v_n,x_n,v_0,x_0,v_1$.

In paths and cycles, we disallow repeated vertices and, therefore, disallow repeated edges. But, in trails and circuits, we only disallow repeated edges, but allow repeated vertices. Also, we have the following observation. 

\begin{lemma}\label{lemma-circuit-cycle}
The set of edges of a directed circuit is the set of edges of a collection of pairwise edge-disjoint directed cycles attached together at shared vertices. 
\end{lemma}

\begin{definition}\label{def-disassemble}
Let $G$ be a directed graph. For a vertex $v$ of $G$, let the in-degree of $v$ be $n$ and out-degree of $v$ be $m$. Recall that $k_v=\max\{m,n\}$. Set $l_v = \min\{m,n\}$. 

To disassemble $G$ at $v$ is to split $v$ into $k_v$ vertices such that
\begin{enumerate}
	\item $l_v$ of these new vertices have in-degree $1$ and out degree $1$.
	\item $k_v-l_v$ of these new vertices have degree $1$ such that 	      
				\begin{itemize}
					\item if $m\geq n$, then each of these degree $1$ vertices has in-degree $0$ and out-degree $1$;
					\item if $m<n$, then each of these degree $1$ vertices has in-degree $1$ and out-degree $0$.
				\end{itemize}
\end{enumerate}
Of course, depending on how inward edges and outward edges are matched, there are many different ways to disassemble $G$ at $v$.

To disassemble $G$ is to disassemble $G$ at all vertices of $G$. Again, there are many ways to disassemble a directed graph. We call each graph resulted from disassembling $G$ a disassembly of $G$ and denote by $\mathrm{Dis}(G)$ the set of all disassemblies of $G$.
\end{definition}

The following lemma is a series of simple observations about disassemblies. 

\begin{lemma}\label{lemma-disassemble}
Let $G$ be a directed graph, and $D$ a disassembly of $G$.
\begin{enumerate}
	\item $D$ is a directed graph whose vertices are all of degree $1$ or $2$. And, each vertex of degree $2$ of $D$ has in-degree $1$ and out-degree $1$.
	\item $D$ is a disjoint union of directed paths and directed cycles.
	\item $E(D)=E(G)$ and there is a natural graph homomorphism from $D$ to $G$ that maps each edge to itself and each vertex $v$ in $D$ the vertex in $G$ used to create $v$.
	\item Under the above natural homomorphism, 	
	\begin{itemize}
		\item each directed path in $D$ is mapped to a directed trail in $G$,
		\item each directed cycle in $D$ is mapped to a directed circuit in $G$,
		\item the collection of all directed cycles in $D$ is mapped to a collection of pairwise edge-disjoint circuits in $G$.
	\end{itemize} 
	\item $\alpha(D) \leq \alpha (G)$ and $\alpha(D) = \alpha (G)$ if and only if the collection of all directed cycles in $D$ is mapped to a collection of $\alpha(G)$ pairwise edge-disjoint directed cycles in $G$ by the natural homomorphism.
\end{enumerate}
\end{lemma}

\begin{proof}
Parts (1-4) are obvious, we leave their proofs to the reader. Part (5) follows from Part (4) and Lemma \ref{lemma-circuit-cycle}.
\end{proof}

\begin{lemma}\label{lemma-circuit-disassembly}
Let $G$ be a directed graph. Given any collection $\mathcal{C}$ of pairwise edge-disjoint directed circuits in $G$, there is a disassembly $D$ of $G$ such that every directed circuit in $\mathcal{C}$ is the image of a directed cycle in $D$ under the natural graph homomorphism from $D$ to $G$.

In particular, there is a disassembly $D$ of $G$ such that $\alpha(D)=\alpha(G)$.
\end{lemma}
\begin{proof}
To get such a $D$, one just needs to make sure that, when disassembling $G$, each pair of adjacent edges in each circuit in $\mathcal{C}$ are matched together. Edges not contained in any circuits in $\mathcal{C}$ can be matched in any possible way when disassembling $G$. 

For the second half of the lemma, just pick $\mathcal{C}$ to be a collection of $\alpha(G)$ pairwise edge-disjoint directed cycles in $G$. Then the conclusion follows from the first half and Part (5) of Lemma \ref{lemma-disassemble}.
\end{proof}

\begin{lemma}\label{lemma-P-disassembly}
Let $G$ be a directed graph. 
\begin{enumerate}
	\item For any disassembly $D$ of $G$, the incidence set $P(D)$ of $D$ is a linear subspace of dimension $\alpha(D) -1$ of $\mathbb{CP}^{|E(G)|-1}$. 
	\item For any two disassemblies $D_1$ and $D_2$ of $G$, $P(D_1)=P(D_2)$ as linear subspaces of $\mathbb{CP}^{|E(G)|-1}$ if and only if, under the natural homomorphisms from $D_1$ and $D_2$ to $G$, the collections of all directed cycles in $D_1$ and $D_2$ are mapped to the same collection of pairwise edge-disjoint circuits in $G$. 
\end{enumerate}
\end{lemma}

\begin{proof}
For Part (1), note that $P(D)$ is determined by the linear equations
\begin{equation}\label{eq-PD}
\begin{cases}
x=0 & \text{if } x \text{ is not an edge of any directed cycle in } D,\\
x=y & \text{if } x \text{ and } y \text{ are edges in the same directed cycle in } D.
\end{cases}
\end{equation}
The solution space of \eqref{eq-PD} in $\C^{|E(G)|}$ is $\alpha(D)$-dimensional since there is a one-to-one correspondence between free variables in the solution and directed cycles in $D$. $P(D)$ is the image of this solution space in $\mathbb{CP}^{|E(G)|-1}$. So $P(D)$ is a linear subspace of dimension $\alpha(D) -1$ of $\mathbb{CP}^{|E(G)|-1}$. 

Now consider Part (2). First, if collections of all directed cycles in $D_1$ and $D_2$ are mapped to the same collection of pairwise edge-disjoint circuits in $G$, then $P(D_1)$ and $P(D_2)$ are defined by the same linear equations. So $P(D_1)=P(D_2)$ in this case. Next, assume that $P(D_1)=P(D_2)$. For an edge $x$ of $G$ and $i=1,2$, $x$ is not contained in any directed cycle in $D_i$ if and only if $x=0$ at all points in $P(D_i)$. Since $P(D_1)=P(D_2)$, this implies that $x$ is not contained in any directed cycle in $D_1$ if and only if $x$ is not contained in any directed cycle in $D_2$. Similarly, edges $x$ and $y$ are contained in the same directed cycle in $D_i$ if and only if $x=y$ at all points in $P(D_i)$ and $x\neq 0$ at some point in $P(D_i)$. Again, since $P(D_1)=P(D_2)$, this implies that $x$ and $y$ are contained in the same directed cycle in $D_1$ if and only if $x$ and $y$ are contained in the same directed cycle in $D_2$. Thus, there is a bijection between directed cycles in $D_1$ and directed cycles in $D_2$ such that corresponding directed cycles have exactly the same edges. This implies that  the collections of all directed cycles in $D_1$ and $D_2$ are mapped to the same collection of pairwise edge-disjoint circuits in $G$ by the natural homomorphisms.
\end{proof}

\subsection{Structure of the incidence set} In this subsection, we give an explicit description of the incidence set and prove Theorem \ref{thm-incidence}.

First, we recall a simple corollary of Vieta's Theorem.

\begin{lemma}\label{lemma-Vieta}
Let $x_1,\dots,x_n$ and $y_1,\dots,y_n$ be two sequences of complex numbers. Then the following statements are equivalent.
\begin{enumerate}
	\item $e_k(x_1,\dots,x_n)=e_k(y_1,\dots,y_n)$ for $k=1,\dots,n$, where $e_k$ is the $k$-th elementary symmetric polynomial.
	\item There is a bijection $\sigma:\{1,\dots,n\}\rightarrow \{1,\dots,n\}$ such that $x_i = y_{\sigma(i)}$ for $i=1,\dots,n$.
\end{enumerate}
\end{lemma}

\begin{lemma}\label{lemma-union-incidence-set}
For every directed graph $G$,
\[
P(G) = \bigcup_{D\in \mathrm{Dis}(G)} P(D)
\] 
as subsets of $\mathbb{CP}^{|E(G)|-1}$.
\end{lemma}

\begin{proof}
For a vertex $v$ of $G$, suppose $x_1,\dots,x_m$ are the edges having $v$ as their initial vertex, and $y_1,\dots,y_n$ are the edges having $v$ as their terminal vertex. Recall that $k_v=\max\{m,n\}$. Define two sequences $In(v)=\{y_{v,1},\dots,y_{v,{k_v}}\}$ and $Out(v)=\{x_{v,1},\dots,x_{v,{k_v}}\}$ by
\begin{itemize}
	\item $y_{v,i}= \begin{cases}
	y_i & \text{if } i \leq n,\\
	0 & \text{if } n<i\leq k_v,
	\end{cases}$
	\item $x_{v,i}= \begin{cases}
	x_i & \text{if } i \leq m,\\
	0 & \text{if } m<i\leq k_v.
	\end{cases}$
\end{itemize}
From the definition of disassemblies, one can see that a point $p$ is in $\bigcup_{D\in \mathrm{Dis}(G)} P(D)$ in and only if that, for every vertex $v$, there is a bijection $\sigma_v: \{1,\dots,k_v\} \rightarrow \{1,\dots,k_v\}$ such that $p$ satisfies $y_{v,i} = x_{v,\sigma_v(i)}$ for all $v\in V(G)$ and $1\leq i \leq k_v$. By Lemma \ref{lemma-Vieta}, this is equivalent to that $p$ satisfies all the incidence relations at all vertices of $G$. In other words, $p\in P(G)$.
\end{proof}

\begin{lemma}\label{lemma-max-P(D)}
Let $G$ be a directed graph, and $D$ a disassembly of $G$. Then $P(D)$ is not a proper subset of $P(D')$ for any $D'\in \mathrm{Dis}(G)$ if and only if the natural homomorphism maps the directed cycles in $D$ to a maximal collection of pairwise edge-disjoint directed cycles in $G$ (in the sense of Definition \ref{def-cycle-spectrum}.)
\end{lemma}

\begin{proof}
For any disassembly $D$ of $G$, denote by $\mathcal{C}(D)$ the collection of pairwise edge-disjoint directed circuits in $G$ given by the images of the directed cycles in $D$ under the natural homomorphism.

First assume that $P(D)$ is not a proper subset of $P(D')$ for any $D'\in \mathrm{Dis}(G)$. If $\mathcal{C}(D)$ is not a maximal collection of pairwise edge-disjoint directed cycles in $G$, then
\begin{enumerate}[1.]
	\item either there is a directed circuit $C'$ in $\mathcal{C}(D)$ that is not a directed cycle,
	\item or there is a directed cycle $C''$ in $G$ that is edge-disjoint from all directed cycles in $\mathcal{C}(D)$.
\end{enumerate}
In Case 1, by Lemma \ref{lemma-circuit-cycle}, one get a collection $\mathcal{C}'$ of pairwise edge-disjoint directed circuits in $G$ by replacing $C'$ with the pairwise edge-disjoint directed cycles whose edges coincide with those of $C'$. By Lemma \ref{lemma-circuit-disassembly}, there is a disassembly $D'$ of $G$ such that $\mathcal{C}(D')$ contains $\mathcal{C}'$. It is easy to see that $P(D)$ is a proper subset of $P(D')$. In Case 2, by Lemma \ref{lemma-circuit-disassembly}, there is a disassembly $D''$ of $G$ such that $\mathcal{C}(D'')$ contains $\mathcal{C}(D) \cup \{C''\}$. It is easy to see that $P(D)$ is a proper subset of $P(D'')$. This shows that $\mathcal{C}(D)$ is a maximal collection of pairwise edge-disjoint directed cycles in $G$.

Now assume that $\mathcal{C}(D)$ is a maximal collection of pairwise edge-disjoint directed cycles in $G$. Let $D'$ be a disassembly of $G$ such that $P(D) \subset P(D')$. Suppose a directed circuit $C$ in $\mathcal{C}(D')$ contains an edge $x$ that is not an edge of any directed cycle in $\mathcal{C}(D)$. Then we know that every point in $P(D')$ satisfies $y=x$ for every edge $y$ of $C$. Since $P(D) \subset P(D')$, this implies that $y=x=0$ in $P(D)$ for every edge $y$ of $C$. Thus, $C$ is edge-disjoint from every directed cycle in $\mathcal{C}(D)$. This contradicts the assumption that $\mathcal{C}(D)$ is a maximal. So the directed circuits in $\mathcal{C}(D')$ do not contain any edge that is not contained in any directed cycle in $\mathcal{C}(D)$. On the other hand, if $x$ is an edge of some directed cycle in $\mathcal{C}(D)$, then $x\neq 0$ at some point in $P(D) \subset P(D')$. This implies that $x$ is the edge of some circuit in $\mathcal{C}(D')$. Altogether, we have:

\emph{Conclusion 1.} The set of edges of all directed circuits in $\mathcal{C}(D')$ is equal to the set of edges of all directed cycles in $\mathcal{C}(D)$.

Let $x$ and $y$ be edges of the same directed circuit in $\mathcal{C}(D')$. Then $x=y$ for all points in $P(D') \supset P(D)$. Thus, $x$ and $y$ are in the same directed cycle in $\mathcal{C}(D)$. Next, assume $x$ and $y$ are edges of different directed circuits $C_x$ and $C_y$ in $\mathcal{C}(D')$ but are in the same directed cycle $C_{x,y}$ in $\mathcal{C}(D)$. Then
\begin{itemize}
	\item for every edge $z$ of $C_x$, $z=x$ at every point in $P(D') \supset P(D)$,
	\item for every edge $w$ of $C_y$, $w=y$ at every point in $P(D') \supset P(D)$,
	\item $x=y$ at every point in $P(D)$.
\end{itemize}
This implies that, for every edge $z$ of $C_x$ and every edge $w$ of $C_y$, $z=x=w=y$ at every point in $P(D)$. Thus, the directed cycle $C_{x,y}$ contains both circuits $C_x$ and $C_y$, which is impossible. Therefore, we have

\emph{Conclusion 2.} Two edges are in the same directed circuit in $\mathcal{C}(D')$ if and only if these two edges are in the same directed cycle in $\mathcal{C}(D)$.

Combining Conclusions 1 and 2, one gets that $\mathcal{C}(D')=\mathcal{C}(D)$. By Part (2) of Lemma \ref{lemma-P-disassembly}, this implies that $P(D)=P(D')$. Thus, $P(D)$ cannot be a proper subset of $P(D')$ for any $D'\in \mathrm{Dis}(G)$.
\end{proof}

\begin{proposition}\label{prop-decomp-incidence-set}
Let $G$ be a directed graph. For every maximal collection $\mathcal{C}$ of pairwise edge-disjoint directed cycles in $G$, there is a disassembly $D_{\mathcal{C}}$ of $G$ such that $\mathcal{C}$ is the collection of images of directed cycles in $D_{\mathcal{C}}$ under the natural homomorphism. 

The set of irreducible components of $P(G)$ is 
\[
\{P(D_{\mathcal{C}})~|~ \mathcal{C} \text{ is a maximal collection of pairwise edge-disjoint directed cycles in $G$}\}. 
\] 
\end{proposition}

\begin{proof}
Again, for any disassembly $D$ of $G$, denote by $\mathcal{C}(D)$ the collection of pairwise edge-disjoint directed circuits in $G$ given by the images of the directed cycles in $D$ under the natural homomorphism.

By Lemma \ref{lemma-circuit-disassembly}, there is a $D_{\mathcal{C}} \in \mathrm{Dis}(G)$ such that $\mathcal{C} \subset \mathcal{C}(D_{\mathcal{C}})$. Since $\mathcal{C}$ is maximal, we have that $\mathcal{C} = \mathcal{C}(D_{\mathcal{C}})$. Consider the union $P:=\bigcup_{\mathcal{C}} P(D_{\mathcal{C}})$, where $\mathcal{C}$ runs through all maximal collection of pairwise edge-disjoint directed cycles in $G$. Let $D \in \mathrm{Dis}(G)$. If $\mathcal{C}(D)$ is a maximal collection of pairwise edge-disjoint directed cycles in $G$, then, by Lemma \ref{lemma-P-disassembly}, $P(D)=P(D_{\mathcal{C}(D)})$ is contained in $P$. If $\mathcal{C}(D)$ is not a maximal collection of pairwise edge-disjoint directed cycles in $G$, then, by Lemma \ref{lemma-max-P(D)}, $P(D)$ is a proper subset of $P(D_{\mathcal{C}})$ for some maximal collection $\mathcal{C}$ of pairwise edge-disjoint directed cycles in $G$. Thus, by Lemma \ref{lemma-union-incidence-set},
\begin{equation}\label{eq-irre-comps-P(G)}
P(G) = P := \bigcup_{\mathcal{C}} P(D_{\mathcal{C}}).
\end{equation}

Let $\mathcal{C}_1$ and $\mathcal{C}_2$ be two distinct maximal collections of pairwise edge-disjoint directed cycles in $G$. By Lemma \ref{lemma-P-disassembly}, $P(D_{\mathcal{C}_1}) \neq P(D_{\mathcal{C}_2})$. By Lemma \ref{lemma-max-P(D)}, $P(D_{\mathcal{C}_1})$ and $P(D_{\mathcal{C}_2})$ cannot be proper subsets of each other. Hence, $P(D_{\mathcal{C}_1})$ and $P(D_{\mathcal{C}_2})$ are not subsets of each other. Since every linear subspace of $\mathbb{CP}^{|E(G)|-1}$ is irreducible, this shows that \eqref{eq-irre-comps-P(G)} is the decomposition of $P(G)$ into its irreducible components. 
\end{proof}

Now the proof of Theorem \ref{thm-incidence} is quite easy.

\begin{proof}[Proof of Theorem \ref{thm-incidence}]
First, by Proposition \ref{prop-decomp-incidence-set}, $P(G)$ is a finite union of linear subspaces of $\mathbb{CP}^{|E(G)|-1}$. In particular, 
\[
\dim P(G) = \max \{\dim P(D_{\mathcal{C}})~|~ \mathcal{C} \text{ is a maximal collection of pairwise edge-disjoint directed cycles in $G$}\}.
\]
By Lemma \ref{lemma-P-disassembly}, $\dim P(D_{\mathcal{C}}) = |\mathcal{C}|-1$, where $|\mathcal{C}|$ is the number of pairwise edge-disjoint directed cycles in the maximal collection $\mathcal{C}$.
So 
\[
\dim P(G) = \max \{|\mathcal{C}|-1~|~ \mathcal{C} \text{ is a maximal collection of pairwise edge-disjoint directed cycles in $G$}\} = \alpha(G)-1.
\]
By the definition of $\gamma_n(G)$ in Definition \ref{def-cycle-spectrum}, one can see that 
\begin{eqnarray*}
\gamma_n(G) & = & |\{\mathcal{C} ~|~ \mathcal{C} \text{ is a maximal collection of pairwise edge-disjoint directed cycles in $G$ satisfying } |\mathcal{C}|=n\}| \\
& = & \text{the number of irreducible components of } P(G) \text{ of dimension } n-1.
\end{eqnarray*}
Recall that the degree of a projective algebraic set is the sum of degrees of its top dimensional irreducible components. Since the degree of any linear subspace of $\mathbb{CP}^{|E(G)|-1}$ is $1$, we have
\begin{eqnarray*}
&& \deg P(G) \\
& = & |\{\mathcal{C} ~|~ \mathcal{C} \text{ is a maximal collection of pairwise edge-disjoint directed cycles in $G$ satisfying } \dim P(D_{\mathcal{C}}) = \alpha(G) -1\}| \\
& = & |\{\mathcal{C} ~|~ \mathcal{C} \text{ is a maximal collection of pairwise edge-disjoint directed cycles in $G$ satisfying } |\mathcal{C}|=\alpha(G)\}| \\
& = & \gamma_{\alpha(G)}(G).
\end{eqnarray*}
\end{proof}

\subsection{Collections pairwise disjoint directed cycles}\label{subsec-disjoint} We prove Corollary \ref{cor-incidence-strong} in this subsection.

Let $G$ be a directed graph. We define a new directed graph $B_G$ by 
\begin{enumerate}
	\item splitting each vertex $v$ of $G$ into two vertices $v_{in}$ and $v_{out}$ such that
				\begin{itemize}
					\item all edges pointing into $v$ in $G$ now point into $v_{in}$ in $B_G$,
					\item all edges pointing out of $v$ in $G$ now point out of $v_{out}$ in $B_G$,
				\end{itemize}
	\item for each vertex $v$ of $G$, adding a single edge $z_v$ to $B_G$ pointing from $v_{in}$ to $v_{out}$. 			
\end{enumerate}
Note that $V(B_G)=\{v_{in}~|~v\in V(G)\} \cup \{v_{out}~|~v\in V(G)\}$ and $E(B_G) = E(G) \cup \{z_v~|~v\in V(G)\}$. Clearly, $B_G$ is a directed bipartite graph since the edges in $E(G)$ all point from $\{v_{out}~|~v\in V(G)\}$ to $\{v_{in}~|~v\in V(G)\}$, and each $z_v$ points from $\{v_{in}~|~v\in V(G)\}$ to $\{v_{out}~|~v\in V(G)\}$. 

For a directed cycle $C$ in $G$ given by the sequence $v_0,x_0,v_1,x_1,\dots,x_{n-1},v_n,x_n,v_0$, we define a directed cycle $\ve(C)$ in $B_G$ by replacing each vertex $v_i$ in $C$ by the sequence $(v_i)_{in},z_{v_i},(v_i)_{out}$. That is, $\ve(C)$ is the directed cycle in $B_G$ given by the sequence 
\[
(v_0)_{in},z_{v_0},(v_0)_{out},x_0,(v_1)_{in},z_{v_1},(v_1)_{out},x_1,\dots,x_{n-1},(v_n)_{in},z_{v_n},(v_n)_{out},x_n,(v_0)_{in}.
\]

\begin{lemma}\label{lemma-extended-cycles}
Let $G$ be a directed graph, and $B_G$, $\ve$ defined as above. Then:
\begin{enumerate}
	\item $\ve$ is a bijection from the set of directed cycles in $G$ to the set of directed cycles in $B_G$;
	\item A collection $\mathcal{C}$ of directed cycles in $G$ is pairwise disjoint if and only if the collection $\{\ve(C)~|~C \in \mathcal{C}\}$ is pairwise edge-disjoint in $B_G$;
	\item A collection $\mathcal{C}$ of pairwise disjoint directed cycles in $G$ is strongly maximal in $G$ if and only if the collection $\{\ve(C)~|~C \in \mathcal{C}\}$ of pairwise edge-disjoint directed cycles is maximal in $B_G$;
	\item $\tilde{\alpha}(G) = \alpha(B_G)$, $\tilde{\gamma}_n(G) = \gamma_n(B_G) ~\forall ~n\in \zed_{>0}$.
\end{enumerate}
\end{lemma}

\begin{proof}
For Part (1), note that $\ve$ is injective since $\ve(C)$ determines the set of edges of $C$. Also, since $B_G$ is bipartite, the edges in any directed cycle in $B_G$ must alternate between $E(G)$ and $\{z_v~|~v\in V(G)\}$. This implies that $\ve$ is surjective.

Two directed cycles $C_1$ and $C_2$ in $G$ share the vertex $v$ if and only if the directed cycles $\ve(C_1)$ and $\ve(C_2)$ in $B_G$ share the edge $z_v$. This implies Parts (2) and (3).

Part (4) follows from Parts (1-3).
\end{proof}

With Lemma \ref{lemma-extended-cycles}, it is easy to see that Corollary \ref{cor-incidence-strong} follows from Theorem \ref{thm-incidence}.

\begin{proof}[Proof of Corollary \ref{cor-incidence-strong}]
We order edges of $G$ as $x_1,\dots,x_n$ and vertices of $G$ as $v_1,\dots,v_m$, where $n=|E(G)|$ and $m=|V(G)|$. This give an ordering of the edges of $B_G$ as $x_1,\dots,x_n,z_{v_1},\dots,z_{v_m}$. 

For a vertex $v$ of $G$, denote by $In(v)$ the set of edges in $G$ pointing into $v$. Define a map 
\[
\Phi:\mathbb{CP}^{|E(G)|-1}=\mathbb{CP}^{n-1} \rightarrow \mathbb{CP}^{|E(B_G)|-1} = \mathbb{CP}^{n+m-1}
\] 
by 
\[
\Phi(x_1:\cdots:x_n)=(x_1:\cdots:x_n:z_{v_1}:\cdots:z_{v_m}),
\] 
where $z_{v_i} = \sum_{x\in In(v_i)} x$.

Let $W$ be the Zariski open set in $\mathbb{CP}^{|E(B_G)|-1}$ given by 
\[
W:= \{(x_1:\cdots:x_n:z_{v_1}:\cdots:z_{v_m}) \in \mathbb{CP}^{|E(B_G)|-1}~|~ (x_1,\dots,x_n) \neq 0\}.
\]
Define a map $\Psi:W \rightarrow \mathbb{CP}^{|E(G)|-1}$ by 
\[
\Psi(x_1:\cdots:x_n:z_{v_1}:\cdots:z_{v_m}) = (x_1:\cdots:x_n).
\]

It is clear that 
\begin{itemize}
	\item $\Phi(\tilde{P}(G)) = P(B_G)$ and $\Psi(P(B_G)) = \tilde{P}(G)$, where $P(B_G)$ is defined as in Definition \ref{def-incidence};
	\item $\phi=\Phi|_{\tilde{P}(G)}:\tilde{P}(G) \rightarrow P(B_G)$ and $\psi=\Psi|_{P(B_G)}:P(B_G) \rightarrow \tilde{P}(G)$ are morphisms of projective algebraic sets;
	\item $\psi\circ \phi =\id_{\tilde{P}(G)}$ and $\phi\circ \psi = \id_{P(B_G)}$.
\end{itemize}
So $\phi$ and $\psi$ are isomorphisms of projective algebraic sets. Therefore, $\tilde{P}(G)$ and $P(B_G)$ are isomorphic as projective algebraic sets. In particular, they have the same dimension and irreducible decomposition. By Part (4) of Lemma \ref{lemma-extended-cycles}, this proves Parts (2) and (4) of Corollary \ref{cor-incidence-strong} follows from Parts (2) and (4) of Theorem \ref{thm-incidence}.

Note that all components of $\Phi$ and $\Psi$ are linear functions of the coordinates. So they map linear subspaces to linear subspaces. By Theorem \ref{thm-incidence}, $P(B_G)$ is the union of finitely many linear subspaces of $\mathbb{CP}^{|E(B_G)|-1}$. So $\tilde{P}(G)$ is a union of finitely many linear subspaces of $\mathbb{CP}^{|E(G)|-1}$. This proves Part (1) of Corollary \ref{cor-incidence-strong}.

Finally, since all irreducible components of $\tilde{P}(G)$ are linear subspaces, and the degree of any linear subspace is $1$, $\deg \tilde{P}(G)$ is equal to the number of irreducible components of $\tilde{P}(G)$ of dimension $\tilde{\alpha}(G)-1 = \alpha(B_G)-1$. So, by Part (4) of Corollary \ref{cor-incidence-strong}, $\deg \tilde{P}(G)=\tilde{\gamma}_{\tilde{\alpha}(G)}(G)$. This proves Part (3) of Corollary \ref{cor-incidence-strong}.
\end{proof}

\subsection{Irreducible incidence sets}
In this subsection, we prove Theorem \ref{thm-Krull-sharp-varieties}. First, we recall a fact about the dimension of projective algebraic sets.

\begin{fact}\label{fact-dim-chain}
The dimension of a projective algebraic set $X$ is the maximal length $r$ of a strictly increasing chain $Z_0 \subsetneq Z_1 \subsetneq \cdots \subsetneq Z_r$ of irreducible projective algebraic set contained in $X$.
\end{fact}

We divide the proof of Theorem \ref{thm-Krull-sharp-varieties} into Lemmas \ref{lemma-variety=>generators} and \ref{lemma-generators=>variety} below.

\begin{lemma}\label{lemma-variety=>generators}
For a directed graph $G$, if $P(G)$ is a projective variety, then $G$ contains exactly $\alpha(G)$ distinct directed cycles, and $P(G)$ is a linear subspace of $\mathbb{CP}^{|E(G)|-1}$.
\end{lemma}

\begin{proof}
Let $\alpha=\alpha(G)$ and $\{C_1,\dots,C_\alpha\}$ a collection of $\alpha$ pairwise edge-disjoint directed cycles in $G$. For each $C_i$, denote by $E(C_i)=\{x_{i,1},\dots,x_{i,l_i}\}$ the set of edges of $C_i$. Consider the linear subspace $L$ of $\mathbb{CP}^{|E(G)|-1}$ defined by 
\begin{eqnarray}
	 \label{eq-in-cycle} x_{i,j}& =& x_{i,k},~\forall ~i,~j \text{ and }k, \\
	 \label{eq-out-cycle} x & = & 0, ~\forall ~x \in E(G)\setminus (\bigcup_{i=1}^{\alpha} E(C_i)).
\end{eqnarray}
Then $L$ is of dimension $\alpha-1$, and is contained in $P(G)$. Assume $L \subsetneq P(G)$. Since $P(G)$ is irreducible, by Fact \ref{fact-dim-chain}, this implies that $\alpha-1 =\dim L<\dim P(G)$, which contradicts Theorem \ref{thm-incidence}. Thus, $P(G)=L$ is a linear subspace of $\mathbb{CP}^{|E(G)|-1}$. It remains to show that $G$ contains exactly $\alpha$ directed cycles. In other words, there are no directed cycles in $G$ other than the cycles $C_1,\dots,C_\alpha$. Assume there is a directed cycle $C$ in $G$ different from $C_1,\dots,C_\alpha$. Consider the equations 
\[
\begin{cases}
x=y, ~\forall ~x,y \in E(C), \\
x=0, ~\forall ~x \in E(G)\setminus E(C),
\end{cases}
\]
where $E(C)$ is the set of edges of $C$. These equations determine a single point $p \in P(G)$. Since $C$ is not equal to any $C_i$, one can see that either $C$ contains an edge in $E(G)\setminus (\bigcup_{i=1}^{\alpha} E(C_i))$, or there is a $C_i$ such that $\emptyset \subsetneq E(C)\cap E(C_i) \subsetneq E(C_i)$. In either case, it is easy to check that $p \notin L$. This is a contradiction since $L=P(G)$.
\end{proof}

\begin{lemma}\label{lemma-generators=>variety}
For a directed graph $G$, if $G$ contains exactly $\alpha(G)$ distinct directed cycles, then $P(G)$ is a linear subspace of $\mathbb{CP}^{|E(G)|-1}$.
\end{lemma}

\begin{proof}
Let $\alpha=\alpha(G)$ and $\{C_1,\dots,C_\alpha\}$ a collection of $\alpha$ pairwise edge-disjoint directed cycles in $G$. For each $C_i$, denote by $E(C_i)=\{x_{i,1},\dots,x_{i,l_i}\}$ the set of edges of $C_i$. We prove that $P(G)$ is equal to the linear subspace $L$ of $\mathbb{CP}^{|E(G)|-1}$ defined by the equations \eqref{eq-in-cycle} and \eqref{eq-out-cycle}.

Since $G$ contains no directed cycles other than $C_1,\dots,C_\alpha$, no edge $x \in E(G)\setminus (\bigcup_{i=1}^{\alpha} E(C_i))$ is contained in any directed cycles in $G$. By Corollary \ref{cor-KR-detect-edge}\footnote{Although Corollary \ref{cor-KR-detect-edge} appears in this paper later than Theorem \ref{thm-Krull-sharp-varieties}, its proof does not depend on Theorem \ref{thm-Krull-sharp-varieties}.}, this means that $x^n \in I(G)$ for some $n >0$. Thus, $P(G)$ satisfies equation \eqref{eq-out-cycle}.

Next consider the common vertices of the directed cycles $C_1,\dots,C_\alpha$.

\emph{Claim: Let $\emptyset \subsetneq S\subset\{C_1,\dots,C_\alpha\}$. Then there is a $C_{i} \in S$ such that $V(C_{i}) \cap (\bigcup_{C \in S\setminus \{C_i\}} V(C))$ is either the empty set or the set of a single vertex, where $V(C)$ is the set of vertices of $C$.}

If this claim is not true, then $|V(C_{i}) \cap (\bigcup_{C \in S\setminus \{C_i\}} V(C))| \geq 2$ for all $C_i \in S$. Based on this, one can inductively construct a sequence $C_{i_1},v_1, C_{i_2}, v_2,\dots v_{n-1},C_{i_n},v_n,\dots$ such that, for all $n \geq 1$
\begin{itemize}
	\item $C_{i_n} \in S$,
	\item $C_{i_n} \neq C_{i_{n+1}}$,
	\item $v_n \in V(C_{i_n}) \cap V(C_{i_{n+1}})$,
	\item $v_n \neq v_{n+1}$.
\end{itemize}
Since $S$ is a finite set, there are positive integers $m$ and $n$ such that $m+1 <n$, $C_{i_n}=C_{i_m}$, and $C_{i_m}, C_{i_{m+1}},\dots, C_{i_{n-1}}$ are pairwise distinct. For $m+1\leq l \leq n-1$, there is a directed path $P_l$ in $C_{i_l}$ from $v_{l-1}$ to $v_l$. And there is a directed path $P_n$ in $C_{i_n}=C_{i_m}$ from $v_{n-1}$ to $v_m$. In the case $v_{n-1}=v_m$, we choose $P_n$ to be the path with no edge. (Note that, if $v_{n-1}=v_m$, then $n>m+2$ since $v_{n-1}\neq v_{n-2}$.) Putting the directed paths $P_{m+1},\dots,P_n$ together, we get a directed circuit $P$ such that $E(P)\neq \emptyset$ and $E(C_i) \nsubseteq E(P)$ for every $i=1,\dots, \alpha$. This implies that $P$ contains a directed cycle other than $C_1,\dots,C_\alpha$, which is a contradiction.

Now we are ready to prove that $P(G)$ satisfies equation \eqref{eq-in-cycle}. Let $S_1 = \{C_1,\dots,C_\alpha\}$. By the above claim, there is an $i_1\in \{1,2,\dots,\alpha\}$ such that $|V(C_{i_1}) \cap (\bigcup_{C \in S\setminus \{C_{i_1}\}} V(C))| \leq 1$. Using this and the fact that $P(G)$ satisfies equation \eqref{eq-out-cycle}, it is easy to see that $P(G)$ satisfies equation \eqref{eq-in-cycle} for $i=i_1$. Now let $S_2 = S_1 \setminus \{C_{i_1}\}$. Applying the above claim to $S_2$, and using the fact $P(G)$ satisfies equation \eqref{eq-out-cycle} and equation \eqref{eq-in-cycle} for $i=i_1$, one can find a $C_{i_2}\in S_2$ such that $P(G)$ satisfies equation \eqref{eq-in-cycle} for $i=i_2$. Then let $S_3 = S_2 \setminus \{C_{i_2}\}$ and iterate the above argument. After $\alpha$ iterations of this argument, we conclude that $P(G)$ satisfies equation \eqref{eq-in-cycle} for $i=1,\dots, \alpha$. Thus, $P(G) \subset L$. But it is clear that $L \subset P(G)$. This shows that $P(G)=L$ is a linear subspace of $\mathbb{CP}^{|E(G)|-1}$. 
\end{proof}

\begin{proof}[Proof of Theorem \ref{thm-Krull-sharp-varieties}]
First consider Part (1). Linear subspaces of a projective space are irreducible. So (1-b) $\Rightarrow$ (1-a). (1-a) $\Rightarrow$ (1-b) and (1-a) $\Rightarrow$ (1-c) follow from Lemma \ref{lemma-variety=>generators}. Finally, Lemma \ref{lemma-generators=>variety} means (1-c) $\Rightarrow$ (1-b). This completes the proof of Part (1) of Theorem \ref{thm-Krull-sharp-varieties}.

By the proof of Corollary \ref{cor-incidence-strong} in Subsection \ref{subsec-disjoint}, $\tilde{P}(G) \cong P(B_G)$ as projective algebraic sets, and the isomorphism preserves linearity. So Part (2) of Theorem \ref{thm-Krull-sharp-varieties} follows from Part (1) of Theorem \ref{thm-Krull-sharp-varieties} and Part (4) of Lemma \ref{lemma-extended-cycles}.

Finally, assume $\tilde{P}(G)$ is a projective variety. Note that $G$ always contains $\alpha(G)$ distinct directed cycles and $\alpha(G)\geq \tilde{\alpha}(G)$. So, if $G$ contains exactly $\tilde{\alpha}(G)$ distinct directed cycles, then $\alpha(G)= \tilde{\alpha}(G)$, and $G$ contains exactly $\alpha(G)$ distinct directed cycles. This means $P(G)$ is also a projective variety. But $\tilde{P}(G)$ is a linear subspace of $P(G)$ of dimension $\tilde{\alpha}(G)-1 =\alpha(G)-1 = \dim P(G)$. By Fact \ref{fact-dim-chain}, this means $P(G)=\tilde{P}(G)$.
\end{proof}

\section{Khovanov-Rozansky Homology of Directed Graphs}\label{sec-KR}

In this section, we review the definition and properties of graded Koszul chain complexes and deduce some basic properties of the Khovanov-Rozansky Homology.

\subsection{Graded Koszul chain complexes} Let $R$ be a $\zed$-graded commutative ring with $1$ and $M$ a $\zed$-graded $R$-module. For any $l\in \zed$, denote by $M\{l\}$ the $\zed$-graded $R$-module obtained from $M$ by shifting the grading by $l$. That is, $M\{l\}=M$ as ungraded $R$-modules, and, for any homogeneous element $m\in M$, $\deg_{M\{l\}} m = \deg_M m +l$.

\begin{definition}\label{def-Koszul}
Let $r$ be a homogeneous element of $R$. The graded Koszul chain complex $C_\ast^R(r)$ over $R$ defined by $r$ is the two term chain complex
\[
C_\ast^R(r) = 0 \rightarrow \underbrace{R\{\deg r\}}_{1} \xrightarrow{r} \underbrace{R}_{0} \rightarrow 0,
\]
where $r$ acts by multiplication and the under-braces indicate the homological grading. Note that this is a chain complex of graded $R$-modules, and its boundary map preserves the module grading.

More generally, the graded Koszul chain complex $C_\ast^R(r_1,r_2,\dots,r_k)$ over $R$ defined by the sequence $\{r_1,r_2,\dots,r_k\}$ of homogeneous elements of $R$ is the tensor product over $R$ of the graded Koszul chain complexes associated to all elements in this sequence, that is, the tensor product
\begin{eqnarray*}
&& C_\ast^R(r_1,r_2,\dots,r_k) = C_\ast^R(r_1) \otimes_R C_\ast^R(r_2) \otimes_R \cdots \otimes_R C_\ast^R(r_k) \\
& = & (0 \rightarrow \underbrace{R\{\deg r_1\}}_{1} \xrightarrow{r_1} \underbrace{R}_{0} \rightarrow 0)\otimes_R (0 \rightarrow \underbrace{R\{\deg r_2\}}_{1} \xrightarrow{r_2} \underbrace{R}_{0} \rightarrow 0)\otimes_R \cdots \otimes_R 0 (\rightarrow \underbrace{R\{\deg r_k\}}_{1} \xrightarrow{r_k} \underbrace{R}_{0} \rightarrow 0).
\end{eqnarray*}
Note that 
\begin{itemize}
	\item $C_\ast^R(r_1,r_2,\dots,r_k)$ is a chain complex of graded $R$-modules, and its boundary map preserves the module grading,
	\item permuting the sequence $\{r_1,r_2,\dots,r_k\}$ permutes the factors in the above tensor product and, therefore, does not change the isomorphism type of the Koszul chain complex $C_\ast^R(r_1,r_2,\dots,r_k)$.
\end{itemize}
When $R$ is clear from the context, we drop it from the notation of Koszul chain complexes.
\end{definition}

Next, we recall some simple facts about graded Koszul chain complexes.

\begin{lemma}\label{lemma-C-0}
$H_0(C_\ast^R(r_1,r_2,\dots,r_k)) \cong R/(r_1,r_2,\dots,r_k)$.
\end{lemma}

\begin{proof}
Note that $C_\ast^R(r_1,r_2,\dots,r_k)$ is of the form
\[
\cdots \rightarrow \underbrace{\left.%
\begin{array}{c}
  R\{ \deg r_1\} \\
  \oplus \\
  \vdots \\
  \oplus \\
	R\{ \deg r_k\}
\end{array}%
\right.}_{1} \xrightarrow{(r_1,\dots,r_k)} \underbrace{R}_{0} \rightarrow 0.
\]
The lemma follows immediately.
\end{proof}

\begin{lemma}\label{lemma-C-ideal}
The endomorphism $C_\ast^R(r_1,r_2,\dots,r_k) \xrightarrow{r_j} C_\ast^R(r_1,r_2,\dots,r_k)$ is a chain map homotopic to $0$ for $j=1,\dots,k$. Consequently, $H_\ast(C_\ast^R(r_1,r_2,\dots,r_k))$ is a finitely generated $R/(r_1,r_2,\dots,r_k)$-module.
\end{lemma}

\begin{proof}
We only prove that $C_\ast^R(r_1,r_2,\dots,r_k) \xrightarrow{r_1} C_\ast^R(r_1,r_2,\dots,r_k)$ is a chain map homotopic to $0$. Define a chain homotopy $H:C_\ast^R(r_1) \rightarrow C_\ast^R(r_1)$ of homological degree $1$ by
\[
\xymatrix{
0  \ar[r] & R\{ \deg r_1\} \ar[r]^>>>>>>{r_1} & R \ar[dl]_{1} \ar[r] & 0  \\
0  \ar[r] & R\{ \deg r_1\} \ar[r]^>>>>>>{r_1} & R\ar[r] & 0.
}
\]
Denote by $d_1$ the boundary map of $C_\ast^R(r_1)$. Then, for the endomorphism $C_\ast^R(r_1) \xrightarrow{r_1} C_\ast^R(r_1)$, we have $r_1 = H\circ d_1 + d_1 \circ H$. So $C_\ast^R(r_1) \xrightarrow{r_1} C_\ast^R(r_1)$ is a chain map homotopic to $0$. Now let $C_\ast^R(r_2,\dots,r_k) \xrightarrow{\id} C_\ast^R(r_2,\dots,r_k)$ be the identity map. For the endomorphism $C_\ast^R(r_1,r_2,\dots,r_k) \xrightarrow{r_1} C_\ast^R(r_1,r_2,\dots,r_k)$, we have $r_1 = (H\otimes \id) \circ d + d \circ (H \otimes \id)$, where $d$ is the boundary map of $C_\ast^R(r_1,r_2,\dots,r_k)$. Thus, $r_1$ is homotopic to $0$.
\end{proof}

Before stating the next lemma, we introduce the standard $R$-basis for $C_\ast^R(r_1,r_2,\dots,r_k)$. For $\ve = 0, 1$ and $j=1,\dots, k$, denote by $1^j_\ve$ the unit $1$ of the copy of $R$ in $C_\ast(r_j)$ at homological degree $\ve$. For $\ve = (\ve_1,\dots,\ve_k) \in \{0,1\}^{\times k}$, define $1_\ve = 1^1_{\ve_1} \otimes \cdots \otimes 1^k_{\ve_k}$. Clearly, $\{ 1_\ve ~|~ \ve \in \{0,1\}^{\times k}\}$ is a basis for the $R$-module $C_\ast^R(r_1,r_2,\dots,r_k)$. We call this basis the standard $R$-basis for $C_\ast^R(r_1,r_2,\dots,r_k)$. The boundary map $d$ of $C_\ast^R(r_1,r_2,\dots,r_k)$ acts on this basis by
\[
d (1_{(\ve_1,\dots,\ve_k)}) = \sum_{\ve_j=1} (-1)^{\sum_{l=1}^{j-1}\ve_l} \cdot r_j \cdot 1_{(\ve_1,\dots,\ve_{j-1}, 0,\ve_{j-1},\dots , \ve_k)}.
\]

\begin{lemma}\label{lemma-contraction-strong}
Let $R$ be a $\zed$-graded commutative ring with $1$ and $x$ a homogeneous indeterminate over $R$. Denote by $\pi_x: R[x] \rightarrow R$ the ring homomorphism given by $\pi_x(f(x))=f(0)$ for all $f(x) \in R[x]$. Let $\{x,f_1,\dots,f_k\}$ be a sequence of homogeneous elements of $R[x]$. Then $C_\ast^{R[x]}(x,f_1,\dots,f_k)$ and $C_\ast^R(\pi_x(f_1),\dots,\pi_x(f_k))$ are homotopic as chain complexes of graded $R$-modules.
\end{lemma}

\begin{proof}
Denote by $\iota: R \rightarrow R[x]$ the standard inclusion $\iota(r)=r$. Then $\pi_x \circ \iota = \id_R$. Define an $R$-module map $\Phi: C_\ast^{R[x]}(x,f_1,\dots,f_k) \rightarrow C_\ast^R(\pi_x(f_1),\dots,\pi_x(f_k))$ by
\[
\Phi(f \cdot 1_{(\ve_0,\ve_1,\dots,\ve_k)}) = \begin{cases}
\pi_x(f) \cdot 1_{(\ve_1,\dots,\ve_k)} & \text{if } \ve_0=0, \\
0 & \text{if } \ve_0=1,
\end{cases}
\] 
where $f \in R[x]$ and $(\ve_0,\ve_1,\dots,\ve_k) \in \{0,1\}^{\times (k+1)}$. Also, define an $R$-module map $\Psi:C_\ast^R(\pi_x(f_1),\dots,\pi_x(f_k)) \rightarrow C_\ast^{R[x]}(x,f_1,\dots,f_k)$ by
\[
\Psi(r \cdot 1_{(\ve_1,\dots,\ve_k)}) = \iota(r) \cdot 1_{(0,\ve_1,\dots,\ve_k)},
\]
where $r \in R$ and $(\ve_1,\dots,\ve_k) \in \{0,1\}^{\times k}$. Then 
\begin{itemize}
	\item $\Phi$ and $\Psi$ are chain maps preserving the homological grading and the $R$-module grading,
	\item $\Phi\circ \Psi = \id_{C_\ast^R(\pi_x(f_1),\dots,\pi_x(f_k))}$,
	\item $C_\ast^{R[x]}(x,f_1,\dots,f_k) = \ker \Phi \oplus \im \Psi$, with $\Psi\circ \Phi|_{\ker \Phi} =0$ and $\Psi\circ \Phi|_{\im \Psi} = \id_{\im \Psi}$.
\end{itemize}
Next, define an $R$-module map $H: C_\ast^{R[x]}(x,f_1,\dots,f_k) \rightarrow C_\ast^{R[x]}(x,f_1,\dots,f_k)$ by 
\begin{eqnarray*}
H(f \cdot 1_{(1,\ve_1,\dots,\ve_k)}) & = & 0, \\
H((r+xf) \cdot 1_{(0,\ve_1,\dots,\ve_k)}) & = & f \cdot 1_{(1,\ve_1,\dots,\ve_k)},
\end{eqnarray*}
where $r \in R$, $f \in R[x]$ and $(\ve_1,\dots,\ve_k) \in \{0,1\}^{\times k}$. It is straightforward to check that
\begin{eqnarray*}
(d\circ H + H \circ d)|_{\ker \Phi} & = & \id_{\ker \Phi}, \\
(d\circ H + H \circ d)|_{\im \Psi} & = & 0.
\end{eqnarray*}
This implies that $\id_{C_\ast^{R[x]}(x,f_1,\dots,f_k)} - \Psi\circ \Phi = d\circ H + H \circ d$. Thus, we have shown that $\Phi$ and $\Psi$ are homotopy equivalences preserving the $R$-module grading. Therefore, $C_\ast^{R[x]}(x,f_1,\dots,f_k)$ and $C_\ast^R(\pi_x(f_1),\dots,\pi_x(f_k))$ are homotopic as chain complexes of graded $R$-modules.
\end{proof}

\subsection{Khovanov-Rozansky homology} First, we give a more ``localized" description of the Khovanov-Rozansky homology of directed graphs, which is closer in spirit to the original definition of the Khovanov-Rozansky homology in knot theory. Then we deduce some basic properties of the Khovanov-Rozansky homology relevant to our goals.

Let $G$ be a directed graph. Recall that the polynomial ring $\zed[E(G)]$ is $\zed$-graded so that every $x \in E(G)$ is homogeneous of degree $1$. For a vertex $v$ of $G$, recall that $E(v)$ is the set of all edges incident at $v$. Write $E(v) = \{x_1,\dots,x_m\}\cup \{y_1,\dots,y_n\}$, where $x_1,\dots,x_m$ are the edges having $v$ as their initial vertex, and $y_1,\dots,y_n$ are the edges having $v$ as their terminal vertex. Recall that $k_v := \max\{m,n\}$ and, for $1\leq l \leq k_v$, $\delta_{v,l}:= e_l(x_1,\dots,x_m) - e_l(y_1,\dots,y_n)$, where $e_l$ is the $l$-th elementary symmetric polynomial.

\begin{definition}\label{def-KR-vertex}
$\mathscr{C}_\ast (v) := C^{\zed[E(v)]}_\ast (\delta_{v,1},\dots,\delta_{v,k_v})$.
\end{definition}

Comparing Definitions \ref{def-KR} and \ref{def-KR-vertex}, we have the following simple lemma.

\begin{lemma}\label{lemma-KR-alt-def}
$\mathscr{C}_\ast (G) = \bigotimes_{v \in V(G)} (\mathscr{C}_\ast (v) \otimes_{\zed[E(v)]} \zed[E(G)])$, where the tensor product ``$\bigotimes_{v \in V(G)}$" is taken over $\zed[E(G)]$.
\end{lemma}

\begin{proof}
Straightforward.
\end{proof}

\begin{lemma}\label{lemma-KR-module}
Let $G$ be a directed graph.
\begin{enumerate}[1.]
	\item $\mathscr{H}_0 (G) \cong \zed[E(G)]/I(G)$, where $I(G)=(\Delta_G)$ is the incidence ideal of $\zed[E(G)]$.
	\item $\mathscr{H}_\ast (G)$ is a finitely generated $\zed[E(G)]/I(G)$-module.
	\item $\mathscr{H}_\ast (G)$ is a finitely generated $\zed$-module if and only if $\mathscr{H}_0 (G)$ is a finitely generated $\zed$-module.
	\item For any subring $R$ of $\zed[E(G)]$, $\ann_R \mathscr{H}_\ast (G) = \ann_R \mathscr{H}_0 (G)$ and $\Kdim_R \mathscr{H}_\ast (G) = \Kdim_R \mathscr{H}_0 (G)$.
\end{enumerate}
\end{lemma}

\begin{proof}
Part 1 follows from Lemma \ref{lemma-C-0}. Part 2 follows from Lemma \ref{lemma-C-ideal}. Part 3 follows from Parts 1 and 2. For Part 4, note that $\ann_{\zed[E(G)]} \mathscr{H}_\ast (G) = \ann_{\zed[E(G)]} \mathscr{H}_0 (G) =I(G)$. So, for any subring $R$ of $\zed[E(G)]$, $\ann_R \mathscr{H}_\ast (G) = \ann_R \mathscr{H}_0 (G) =I(G) \cap R$ and $\Kdim_R \mathscr{H}_\ast (G) = \Kdim_R \mathscr{H}_0 (G) = \Kdim_R R/(I(G)\cap R)$.
\end{proof}

The following lemma is obvious.

\begin{lemma}\label{lemma-edge-removal}
Let $G$ be a directed graph and $\hat{G}$ the directed graph obtained from $G$ by removing a single edge $x\in E(G)$. Under the standard identification $\zed[E(\hat{G})]\cong\zed[E(G)]/(x)$, 
\[
\mathscr{C}_\ast(\hat{G}) \otimes C_\ast \cong \mathscr{C}_\ast(G)/x\cdot \mathscr{C}_\ast(G)
\] 
as graded Koszul chain complexes over $\zed[E(\hat{G})]$. Here, $C_\ast$ is a simple Koszul chain complex over $\zed[E(\hat{G})]$ of the form
\[
C_\ast := 
\begin{cases}
0 \rightarrow \underbrace{\zed[E(\hat{G})]}_{0} \rightarrow 0 & \text{if removing } x \text{ does not change } k_v \text{ for any } ~v \in V(G), \\
C_\ast^{\zed[E(\hat{G})]}(0)_{k_v} & \text{if removing } x \text{ reduces } k_v \text{ by } 1 \text{ for a single } v \in V(G), \\
C_\ast^{\zed[E(\hat{G})]}(0)_{k_u} \otimes_{\zed[E(\hat{G})]} C_\ast^{\zed[E(\hat{G})]}(0)_{k_v} & \text{if removing } x \text{ reduces } k_u \text{ and } k_v  \text{ by } 1 \text{ for two } u, ~v \in V(G),
\end{cases}
\]
where $C_\ast^{\zed[E(\hat{G})]}(0)_{k_v} := 0 \rightarrow \underbrace{\zed[E(\hat{G})]\{k_v\}}_{1} \xrightarrow{0} \underbrace{\zed[E(\hat{G})]}_{0} \rightarrow 0$.

In particular, $\mathscr{H}_0 (\hat{G}) \cong \mathscr{H}_0 (G) /x \mathscr{H}_0 (G)$ as graded $\zed[E(G)]$-modules under the standard identification $\zed[E(\hat{G})]\cong\zed[E(G)]/(x)$.
\end{lemma}

\begin{proof}
Straightforward.
\end{proof}

Lemmas \ref{lemma-end-removal} and \ref{lemma-vertex-2-removal} below follow from Lemma \ref{lemma-contraction-strong}.

\begin{lemma}\label{lemma-end-removal}
Let $G$ be a directed graph, $v$ a vertex of $G$ of degree $1$ and $x$ the edge incident at $v$. Denote by $\hat{G}$ the directed graph obtained from $G$ by removing $v$ and $x$. Then, under the standard identification $\zed[E(G)]=\zed[E(\hat{G})][x]$, $\mathscr{C}_\ast(G) \simeq  \mathscr{C}_\ast(\hat{G})$ as chain complexes of graded $\zed[E(\hat{G})]$-modules.

In particular, $\mathscr{H}_\ast (G) \cong \mathscr{H}_\ast (\hat{G})$ as $\zed \oplus \zed$-graded $\zed[E(G)]$-modules under the standard identification $\zed[E(\hat{G})]\cong\zed[E(G)]/(x)$.
\end{lemma}

\begin{proof}
Since $\deg v =1$, $\pm x$ is an incidence relation. Apply Lemma \ref{lemma-contraction-strong} to this entry.
\end{proof}

\begin{lemma}\label{lemma-vertex-2-removal}
Let $G$ be a directed graph. Assume $v$ is a vertex of $G$ of degree $2$ such that an edge $x$ has $v$ as its initial vertex, an edge $y$ has $v$ as its terminal vertex and $x \neq y$. Denote by $\hat{G}$ the directed graph obtained from $G$ by removing $v$ and merging $x,~y$ into a single directed edge. Then, under the standard identification $\zed[E(G)]=\zed[E(\hat{G})][x-y]$, $\mathscr{C}_\ast(G) \simeq  \mathscr{C}_\ast(\hat{G})$ as chain complexes of graded $\zed[E(\hat{G})]$-modules.

In particular, $\mathscr{H}_\ast (G) \cong \mathscr{H}_\ast (\hat{G})$ as $\zed \oplus \zed$-graded $\zed[E(G)]$-modules under the standard identification $\zed[E(\hat{G})]\cong\zed[E(G)]/(x-y)$.
\end{lemma}

\begin{proof}
Since $\deg v =2$, $x-y$ is an incidence relation. Apply Lemma \ref{lemma-contraction-strong} to this entry.
\end{proof}

\begin{lemma}\label{lemma-fuse-vertex}
Let $G$ be a directed graph, and $u$, $v$ two distinct vertices of $G$. Denote by $G_{u \# v}$ the directed graph obtained from $G$ by identifying $u$ and $v$ into a single vertex. Note that there is a natural one-to-one correspondence between $E(G)$ and $E(G_{u \# v})$, which gives an identification $\zed[E(G)]=\zed[E(G_{u \# v})]$. Under this identification, there is a surjective $\zed[E(G)]$-module homomorphism $\mathscr{H}_0 (G_{u \# v}) \rightarrow \mathscr{H}_0 (G)$.
\end{lemma}

\begin{proof}
Denote by $u \# v$ the vertex in $G_{u \# v}$ that is the image of $u,~v$. Define $X_u$ (resp. $X_v$) to be the set of edges in $G$ having $u$ (resp. $v$) as their initial vertex, and $Y_u$ (resp. $Y_v$) to be the set of edges in $G$ having $u$ (resp. $v$) as their terminal vertex. Denote by $e_l(X)$ the $l$-th elementary symmetric polynomial in $X$ for any finite set $X$ of variables. Then  
\begin{eqnarray*}
\delta_{u,l} & = & e_l(X_u) - e_l (Y_u), \\
\delta_{v,l} & = & e_l(X_v) - e_l (Y_v), \\
\delta_{u\#v,l} & = & e_l(X_u\cup X_v) - e_l (Y_u\cup Y_v).
\end{eqnarray*}
Note that $e_l(X \cup Y) = \sum_{i=0}^l e_i(X) e_{l-i}(Y)$ for any sets $X$, $Y$ of variables. So
\begin{eqnarray*}
\delta_{u\#v,l} & = & e_l(X_u\cup X_v) - e_l (Y_u\cup Y_v) \\
& = & \sum_{i=0}^l (e_i(X_u) e_{l-i}(X_v) - e_i(Y_u) e_{l-i}(Y_v)) \\
& = & \sum_{i=0}^l \delta_{u,i}e_{l-i}(X_v) + e_i(Y_u)\delta_{v,l-i}.
\end{eqnarray*}
This implies that $\Delta_(G_{u\# v}) \subset I(G)$ and, therefore, $I(G_{u\# v}) \subset I(G)$, where $I(G)$ is the incidence ideal for $G$ of $\zed[E(G)]$. By Lemma \ref{lemma-KR-module}, $\mathscr{H}_0 (G) \cong \zed[E(G)]/I(G)$ and $\mathscr{H}_0 (G_{u\# v}) \cong \zed[E(G)]/I(G_{u\# v})$. Thus, the identity map $\zed[E(G)] \xrightarrow{\id} \zed[E(G)]$ induces a surjective $\zed[E(G)]$-module homomorphism $\mathscr{H}_0 (G_{u \# v}) \rightarrow \mathscr{H}_0 (G)$.
\end{proof}

\section{Krull Dimension and Directed Cycles}\label{sec-Krull-cycles}

We prove Theorems \ref{thm-KR-detect}, \ref{thm-KR-detect-vertex}, Proposition \ref{prop-Krull-bound-E} and Corollary \ref{cor-KR-detect-edge} in this section.

\subsection{Properties of Krull dimension} We start by collecting properties of the Krull dimension relevant to our proofs. We prove simple properties and give references to more complex ones.

The first property is an obvious fact.

\begin{lemma}\label{lemma-Krull-quotient}
Suppose that $R$ and $R'$ are commutative rings with $1$ and that $\psi:R \rightarrow R'$ is a surjective homomorphism of commutative rings with $1$. We have the following.

\begin{enumerate}[1.]
	\item An ideal $I'$ of $R'$ is prime if and only if $\psi^{-1}(I')$ is a prime ideal of $R$.
	\item $\Kdim R \geq \Kdim R'$.
	\item If $M$ and $M'$ are $R$-modules and there is a surjective $R$-module homomorphism $M\rightarrow M'$, then $\ann_R(M) \subset \ann_R(M')$ and $\Kdim_R M \geq \Kdim_R M'$. 
\end{enumerate}
\end{lemma}

\begin{proof}
For Part 1, recall that $I'$ is prime if and only if $ab \in I'$ implies that $a \in I'$ or $b \in I'$ for any $a, b \in R'$. This property is clearly preserved by pull-backs of ideals. Using that $\psi$ is surjective, it is also easy to verify that $I'$ is prime if $\psi^{-1}(I')$ is prime.

For Part 2, note that, by Part 1, any chain of prime ideals $\mathfrak{p}_0' \subsetneq \mathfrak{p}_1'\subsetneq \cdots \subsetneq\mathfrak{p}_n' \subsetneq R'$ is pulled back to a chain of prime ideals $\psi^{-1}(\mathfrak{p}_0') \subsetneq \psi^{-1}(\mathfrak{p}_1')\subsetneq \cdots \subsetneq \psi^{-1}(\mathfrak{p}_n') \subsetneq R$. This implies that $\Kdim R \geq \Kdim R'$.

For Part 3, $\ann_R(M) \subset \ann_R(M')$ is obvious. So there is a surjective ring homomorphism $R/\ann_R(M) \rightarrow R/\ann_R(M')$. Thus, by Part 2, $\Kdim_R M = \Kdim R/\ann_R(M) \geq \Kdim R/\ann_R(M') = \Kdim_R M'$.
\end{proof}

\begin{lemma}\label{lemma-fg-extension}
Suppose that $R$ is a commutative ring with $1$ and $R'$ is a subring of $R$ containing $1$. If $R$ is a finitely generated $R'$-module, then $\Kdim R=\Kdim R'$.
\end{lemma}

\begin{proof}
Since $R$ is a finitely generated $R'$-module, $R$ is an integral extension of $R'$ by \cite[Corollary 5.23]{Rowen}. But, by \cite[Corollary 6.33]{Rowen}, integral extensions do not change Krull dimensions. So the lemma follows. For more details, see \cite[Chapters 5 and 6]{Rowen}.
\end{proof}

The following lemma is well-known. See for example \cite{BMRH}.

\begin{lemma}\label{lemma-Krull-polynomial}
If $R$ is a commutative Noetherian ring with $1$, then $\Kdim R[x] = \Kdim R +1$. In particular, $\Kdim \zed[x_1,\dots,x_n]=n+1$ and $\Kdim \mathbb{F}[x_1,\dots,x_n]=n$ for any field $\mathbb{F}$.
\end{lemma}

\begin{definition}\label{def-codim}
Let $R$ be a commutative ring with $1$ and $\mathfrak{p}$ a prime ideal of $R$. Then the codimension (also known as height) of $\mathfrak{p}$ is defined to be
\[
\codim \mathfrak{p}:= \sup\{n\geq 0 ~| \text{ there are } n+1 \text{ prime ideals } \mathfrak{p}_0, \mathfrak{p}_1,\dots,\mathfrak{p}_n \text{ of } R \text{ such that } \mathfrak{p}_0 \subsetneq \mathfrak{p}_1\subsetneq \cdots \subsetneq\mathfrak{p}_n = \mathfrak{p}\}.
\]
\end{definition}

The next lemma about codimension is featured in many textbooks on commutative algebra.

\begin{lemma}\label{lemma-ideal-codim-dim}
Let $R$ be a commutative ring with $1$, and $\mathfrak{p}$ a prime ideal of $R$. Then $\Kdim R \geq \Kdim R/\mathfrak{p} + \codim \mathfrak{p}$.
\end{lemma}

\begin{proof}
Let $\mathfrak{p}_0 \subsetneq \mathfrak{p}_1\subsetneq \cdots \subsetneq\mathfrak{p}_c = \mathfrak{p}$ be a chain of prime ideals in $R$, where $c$ is the codimension of $\mathfrak{p}$ in $R$. Since $\mathfrak{p}$ is prime, $R/\mathfrak{p}$ is a domain and, therefore, the zero ideal of $R/\mathfrak{p}$ is prime. Set $d= \Kdim R/\mathfrak{p}$. Then there exists a chain $0=\mathfrak{q}_0 \subsetneq \mathfrak{q}_1 \subsetneq \cdots \subsetneq \mathfrak{q}_d \subsetneq R/\mathfrak{p}$ of prime ideals in $R/\mathfrak{p}$. By Lemma \ref{lemma-Krull-quotient}, we get a chain $\mathfrak{p}_0 \subsetneq \mathfrak{p}_1\subsetneq \cdots \subsetneq\mathfrak{p}_c = \mathfrak{p}\subsetneq \pi^{-1}(\mathfrak{q}_1) \subsetneq \cdots \subsetneq \pi^{-1}(\mathfrak{q}_d) \subsetneq R$ of prime ideals in $R$, where $\pi:R \rightarrow R/\mathfrak{p}$ is the standard quotient map. So $\Kdim R \geq \Kdim R/\mathfrak{p} + \codim \mathfrak{p}$.
\end{proof}

\begin{lemma}\label{lemma-Krull-base-change}
Let $I$ be a homogeneous ideal of $\zed[x_1,\dots,x_n]$ contained in the homogeneous ideal $(x_1,\dots,x_n)$ of $\zed[x_1,\dots,x_n]$. Set
\begin{eqnarray*}
R_\zed & := & \zed[x_1,\dots,x_n] / I \\
R_\Q & := & \Q \otimes_\zed R_\zed = \Q[x_1,\dots,x_n] / \Q \otimes_\zed I.
\end{eqnarray*}
We have:
\begin{enumerate}[1.]
	\item Let $\mathfrak{p}_\zed := (x_1,\dots,x_n) \subset R_\zed$ and $\mathfrak{p}_\Q := \Q \otimes_\zed \mathfrak{p}_\zed = (x_1,\dots,x_n) \subset R_\Q$. Then $\mathfrak{p}_\zed $ and $\mathfrak{p}_\Q$ are both prime ideals, and $\codim \mathfrak{p}_\zed = \codim \mathfrak{p}_\Q$.
	\item $\Kdim R_\zed \geq \Kdim R_\Q +1$.
\end{enumerate}
\end{lemma}

\begin{proof}
For Part 1, note that $R_\zed/\mathfrak{p}_\zed \cong \zed[x_1,\dots,x_n]/(x_1,\dots,x_n) \cong \zed$ and $R_\Q/\mathfrak{p}_\Q \cong \Q[x_1,\dots,x_n]/(x_1,\dots,x_n) \cong \Q$. So $\mathfrak{p}_\zed $ and $\mathfrak{p}_\Q$ are both prime ideals. 

Denote by $\jmath:R_\zed \rightarrow R_\Q$ the ring homomorphism given by $\jmath(r) = 1\otimes r$.

Set $c= \codim \mathfrak{p}_\Q$ and pick a chain $\mathfrak{p}_0 \subsetneq \mathfrak{p}_1\subsetneq \cdots \subsetneq\mathfrak{p}_c = \mathfrak{p}_\Q$ of prime ideals in $R_\Q$. Then $\jmath^{-1}(\mathfrak{p}_0) \subsetneq \jmath^{-1}(\mathfrak{p}_1) \subsetneq \cdots \subsetneq \jmath^{-1}(\mathfrak{p}_c) = \mathfrak{p}_\zed$ is a chain of prime ideals in $R_\zed$. So $\codim \mathfrak{p}_\zed \geq \codim \mathfrak{p}_\Q$.

Set $d= \codim \mathfrak{p}_\zed$ and pick a chain $\mathfrak{q}_0 \subsetneq \mathfrak{q}_1 \subsetneq \cdots \subsetneq \mathfrak{q}_d = \mathfrak{p}_\zed$ of prime ideals in $R_\zed$. Denote by $(\jmath(\mathfrak{q}_i))$ the ideal of $R_\Q$ generated by $\jmath(\mathfrak{q}_i)$, which is prime since $\mathfrak{q}_i$ is prime. Moreover, we have that $(\jmath(\mathfrak{q}_i)) \subsetneq (\jmath(\mathfrak{q}_{i+1}))$. Otherwise, there would be $g\in \mathfrak{q}_{i+1} \setminus \mathfrak{q}_i$ and $l\in \zed$ such that $l>0$ and $lg \in \mathfrak{q}_i$. But $\mathfrak{q}_i$ is prime. This means $l \in \mathfrak{q}_i \subset \mathfrak{p}_\zed$, which is a contradiction. So $(\jmath(\mathfrak{q}_0) )\subsetneq (\jmath(\mathfrak{q}_1)) \subsetneq \cdots \subsetneq (\jmath(\mathfrak{q}_d)) = \mathfrak{p}_\Q$ is a chain of prime ideals in $R_\Q$. Thus, $\codim \mathfrak{p}_\zed \leq \codim \mathfrak{p}_\Q$. This completes the proof of Part 1.

For Part 2, note that $\mathfrak{p}_\Q$ is the unique homogeneous maximal ideal of $R_\Q$. So, by \cite[Corollary 13.7]{Eisenbud}, $\Kdim R_\Q = \codim \mathfrak{p}_\Q$. By Lemma \ref{lemma-ideal-codim-dim}, $\Kdim R_\zed \geq \Kdim R_\zed/\mathfrak{p}_\zed + \codim \mathfrak{p}_\zed = \Kdim \zed + \codim \mathfrak{p}_\Q = \Kdim R_\Q +1$.
\end{proof}

\begin{corollary}\label{cor-Krull-base-change-graph}
For any directed graph $G$ and any subset $E$ of $E(G)$, 
\[
\Kdim_{\Q[E]} \mathscr{H}_0^\Q (G) \leq \Kdim_{\zed[E]} \mathscr{H}_0 (G) -1.
\]
\end{corollary}

\begin{proof}
Apply Lemma \ref{lemma-Krull-base-change} to the ideal $\ann_{\zed[E]}\mathscr{H}_0 (G)=\zed[E]\cap I(G)$ of $\zed[E]$, where $I(G)$ is the incidence ideal of $\zed[E(G)]$.
\end{proof}

\begin{remark}
It is not hard to check that any homogeneous maximal ideal of $R_\zed$ in Lemma\ref{lemma-Krull-base-change} is of the form $\mathfrak{m}= \mathfrak{p}_\zed + (p) = \mathfrak{p}_\zed \oplus p \cdot \zed$, where $p \in \zed$ is a prime number. By \cite[Corollary 13.7]{Eisenbud}, this implies that 
\[
\Kdim  R_\zed = \max \{\codim (\mathfrak{p}_\zed + (p)) ~|~ p \text{ is a prime integer}\}.
\] 
So, to prove Conjecture \ref{conj-Q-Z-same}, one just needs to verify that $\codim (\mathfrak{p}_\zed + (p)) = \codim \mathfrak{p}_\zed +1$ for the ring $R_\zed = {\zed[E]/(\zed[E]\cap I(G))}$ for any directed graph $G$, any subset $E$ of $E(G)$ and any prime integer $p$.
\end{remark}

\subsection{Directed cycles} We prove Theorem \ref{thm-KR-detect} and Proposition \ref{prop-Krull-bound-E} in this subsection. Our proof starts with the following simple observation.

\begin{lemma}\label{lemma-fg}
Let $R$ be a $\zed$-graded commutative ring with $1$, and $R_n$ the homogeneous component of $R$ of degree $n$. Assume that
\begin{itemize}
	\item $R_n=0$ if $n<0$,
	\item $R_n$ is a finitely generated $R_0$-module for each $n \geq 0$.
\end{itemize}
Suppose that $a_1,\dots,a_k$ are homogeneous elements of $R$ of positive degrees. Let $I=(a_1,\dots,a_k)$ be the ideal of $R$ generated by $a_1,\dots,a_k$, and $R'=R_0[a_1,\dots,a_k]$ the $R_0$-subalgebra of $R$ generated by $a_1,\dots,a_k$. Then $R$ is a finitely generated $R'$-module if and only if $R/I$ is a finitely generated $R_0$-module.
\end{lemma}

\begin{proof}
Assume that $R$ is a finitely generated $R'$-module. Then there exist $r_1,\dots,r_m \in R$ such that, for every $r \in R$, there are $f_1,\dots,f_m \in R_0[x_1,\dots,x_k]$ satisfying $r = \sum_{i=1}^m f_i(a_1,\dots,a_k) \cdot r_i$. So $r+I = (\sum_{i=1}^m f_i(a_1,\dots,a_k) \cdot r_i) + I = \sum_{i=1}^m f_i(0,\dots,0) \cdot (r_i +I)$. Note that $f_i(0,\dots,0) \in R_0$. This shows that $R/I$ is generated over $R_0$ by $\{r_1+I,\dots,r_m+I\}$. So $R/I$ is a finitely generated $R_0$-module.

Now assume that $R/I$ is a finitely generated $R_0$-module. Then, there is an $N>0$ such that $R_n \subset I$ if $n>N$. Since each $R_n$ is finitely generated over $R_0$, there is a finite generating set $\{s_1,\dots,s_l\}$ of homogeneous elements for the $R_0$-module $\bigoplus_{n=0}^N R_n$. We claim that $\{s_1,\dots,s_l\}$ is also a generating set for the $R'$-module $R$. To show this, we only need to show that $R_n \subset \sum_{j=1}^l R'\cdot s_j$. For $n \leq N$, this is trivial since $R_n \subset \sum_{j=1}^l R_0\cdot s_j$ for $n \leq N$. Now assume that $R_n \subset \sum_{j=1}^l R'\cdot s_j$ for all $n \leq K$, where $K \geq N$. Then $R_{K+1} \subset I$. So $r = \sum_{i=1}^k r_i a_i$ for any $r \in R_{K+1}$, where each $r_i$ is of the degree $K+1- \deg a_i \leq K$. So $r_i \in  \sum_{j=1}^l R'\cdot s_j$ for each $1\leq i \leq k$. Thus, $r \in \sum_{j=1}^l R'\cdot s_j$. This shows that $R_{K+1} \subset \sum_{j=1}^l R'\cdot s_j$. Thus, $R$ is a finitely generated $R'$-module.
\end{proof}

\begin{corollary}\label{cor-fg-sym}
Let $R$ be a $\zed$-graded commutative ring with $1$, and $R_n$ the homogeneous component of $R$ of degree $n$. Assume that
\begin{itemize}
	\item $R_n=0$ if $n<0$,
	\item $R_n$ is a finitely generated $R_0$-module for each $n \geq 0$.
\end{itemize}
Suppose that $a_1,\dots,a_k$ are homogeneous elements of $R$ of degree $1$. For $1\leq l \leq k$, let $e_l = e_l(a_1,\dots,a_k)=\sum_{1\leq i_1< \cdots < i_l \leq k} a_{i_1}\cdots a_{i_l}$. Then $R/(a_1,\dots,a_k)$ is finitely generated over $R_0$ if and only if $R/(e_1,\dots,e_k)$ is finitely generated over $R_0$.
\end{corollary}

\begin{proof}
Consider the $R_0$-subalgebras $R'=R_0[a_1,\dots,a_k]$ and $R''=R_0[e_1,\dots,e_k]$ of $R$. Denote by $\Sym(x_1,\dots,x_k)$ the ring of symmetric polynomials in variables $x_1,\dots,x_k$ over $R_0$. Then $R_0[x_1,\dots,x_k]$ is a $\Sym(x_1,\dots,x_k)$-module generated by $k!$ generators. This implies that $R'$ is a finitely generated $R''$-module. Thus, $R$ is a finitely generated $R'$-module if and only if $R$ is a finitely generated $R''$-module. Now the corollary follows from Lemma \ref{lemma-fg}.
\end{proof}

\begin{corollary}\label{cor-fg-acyclic}
Let $G$ be an acyclic directed graph. Then $\mathscr{H}_0(G)$ is a finitely generated $\zed$-module.
\end{corollary}

\begin{proof}
We induct on $|V(G)|$. If $G$ only has one vertex, then it has no edges since it is acyclic. So $\mathscr{H}_0(G)$ is just the base ring $\zed$. Now assume the corollary is true for any graph with $n$ vertices and $G$ is an acyclic directed graph with $n+1$ vertices. Since $G$ is acyclic, there is a vertex $v$ of $G$ that is not the terminal vertex of any edge. Let $x_1,\dots,x_k$ be the edges having $v$ as their initial vertex. Denote by $G'$ the directed graph obtained from $G$ by removing $v$ and $x_1,\dots,x_k$. Then $G'$ is an acyclic directed graph with $n$ vertices. So $\mathscr{H}_0(G')$ is a finitely generated $\zed$-module. Let $I$ be the ideal of $\zed[E(G)]$ generated by the set $\{\delta_{u,l}~|~ u\in V(G)\setminus \{v\},~1\leq l \leq k_u\}$ of all incidence relations of $G$ except those at $v$. Set $R=\zed[E(G)]/I$. Then $\mathscr{H}_0(G) \cong R/(e_1(x_1,\dots,x_k),\dots,e_k(x_1,\dots,x_k))$ and, by Lemma \ref{lemma-edge-removal}, $\mathscr{H}_0(G') \cong R/(x_1,\dots,x_k)$. Thus, by Corollary \ref{cor-fg-sym}, $\mathscr{H}_0(G)$ is also a finitely generated $\zed$-module.
\end{proof}

Next, we prove Proposition \ref{prop-Krull-bound-E}.

\begin{proof}[Proof of Proposition \ref{prop-Krull-bound-E}]
By Corollary \ref{cor-Krull-base-change-graph}, we have that $\Kdim_{\Q[E]} \mathscr{H}_0^\Q (G) \leq \Kdim_{\zed[E]} \mathscr{H}_0 (G) -1$. It remains to show that $\alpha_E(G) \leq \Kdim_{\Q[E]} \mathscr{H}_0^\Q (G)$. Fix a collection $\mathcal{C}$ of $\alpha_E(G)$ pairwise edge-disjoint directed cycles in $G$, each of which contains at least one edge in $E$. Denote by $G'$ the subgraph of $G$ consisting of exactly the edges and vertices in $\mathcal{C}$. Then, by Lemma \ref{lemma-edge-removal}, $\mathscr{H}_0^\Q(G') \cong \mathscr{H}_0^\Q(G)/I\cdot\mathscr{H}_0^\Q(G)$, where $I$ is the ideal of $\Q[E(G)]/(\Q\otimes_\zed I(G))$ generated by all the edges in $E(G)\setminus E(G')$. Here, $I(G)$ is the incidence ideal of $\zed[E(G)]$. So, by Lemma \ref{lemma-Krull-quotient}, we have $\Kdim_{\Q[E]} \mathscr{H}_0^\Q (G) \geq \Kdim_{\Q[E]} \mathscr{H}_0^\Q (G')$. For any vertex $v$ of $G'$ of degree $2k$, split it into $k$ vertices of degree $2$, each of which is the vertex of two edges in the same directed cycle in the collection $\mathcal{C}$. This changes $G'$ to a new directed graph $G''$ consisting of $\alpha_E(G)$ pairwise disjoint directed cycles, each of which contains at least one edge in $E$. By Lemmas \ref{lemma-fuse-vertex} and \ref{lemma-Krull-quotient}, $\Kdim_{\Q[E]} \mathscr{H}_0^\Q (G') \geq \Kdim_{\Q[E]} \mathscr{H}_0^\Q (G'')$. By the construction of $G''$ and Lemma \ref{lemma-vertex-2-removal}, one can see that $\ann_{\Q[E]} \mathscr{H}_0 (G'')$ is the ideal of $\Q[E]$ generated by the set 
\[
(E\setminus E(G'')) \cup \{x-y ~|~ x,~y \in E \cap E(G'') \text{ are contained in the same directed cycle in } \mathcal{C} \}.
\]
Thus, $\Q[E]/\ann_{\Q[E]} \mathscr{H}_0 (G'')$ is isomorphic to a polynomial ring over $\Q$ of $\alpha_E(G)$ variables. By Lemma \ref{lemma-Krull-polynomial}, 
\[
\Kdim_{\Q[E]} \mathscr{H}_0^\Q (G'') = \Kdim \Q[E]/\ann_{\Q[E]} \mathscr{H}_0 (G'') = \alpha_E(G).
\]
Therefore, $\Kdim_{\Q[E]} \mathscr{H}_0^\Q (G) \geq \Kdim_{\Q[E]} \mathscr{H}_0^\Q (G') \geq \Kdim_{\Q[E]} \mathscr{H}_0^\Q (G'') =\alpha_E(G)$.
\end{proof}

\begin{lemma}\label{lemma-Krull-beta}
For any directed graph $G$, $\Kdim_{\zed[E(G)]} \mathscr{H}_0 (G) \leq \beta(G)+1$.
\end{lemma}

\begin{proof}
Let $k = \beta(G)$, and $x_1,\dots,x_k$ $k$ edges of $G$, whose removal from $G$ yields an acyclic directed graph $G'$. Thus,by Corollary \ref{cor-fg-acyclic}, $\zed[E(G)]/(I(G)+(x_1,\dots,x_k)) \cong \zed[E(G')]/I(G') \cong \mathscr{H}_0 (G')$ is a finitely generated $\zed$-module. Here, $I(G)$ is the incidence ideal of $\zed[E(G)]$. So, by Lemma \ref{lemma-fg}, $\zed[E(G)]/I(G)$ is a finitely generated module over its subring $R$ generated over $\zed$ by $x_1,\dots,x_k$. By Lemma \ref{lemma-fg-extension}, $\Kdim_{\zed[E(G)]} \mathscr{H}_0 (G) = \Kdim \zed[E(G)]/I(G) = \Kdim R$. But $R$ is a quotient ring of the polynomial ring $\zed[x_1,\dots,x_k]$. So, by Lemmas \ref{lemma-Krull-quotient} and \ref{lemma-Krull-polynomial}, $\Kdim R \leq \Kdim \zed[x_1,\dots,x_k] = k+1$. Therefore, $\Kdim_{\zed[E(G)]} \mathscr{H}_0 (G) \leq k+1 = \beta(G)+1$.
\end{proof}

We have effectively proved Theorem \ref{thm-KR-detect}. Here we recap what was done.

\begin{proof}[Proof of Theorem \ref{thm-KR-detect}]
In Part 1, (1) $\Rightarrow$ (2) follows from Corollary \ref{cor-fg-acyclic}. (2) $\Leftrightarrow$ (3) follows from Lemma \ref{lemma-KR-module}. (2) $\Leftrightarrow$ (4) is elementary. (2) $\Rightarrow$ (5) follows from Lemma \ref{lemma-fg-extension}. (5) $\Rightarrow$ (6) follows from Corollary \ref{cor-Krull-base-change-graph}. (6) $\Rightarrow$ (1) follows from Proposition \ref{prop-Krull-bound-E}. This completes the proof of Part 1.

Part 2 follows from Theorem \ref{thm-incidence}, Fact \ref{fact-kdim}, Corollary \ref{cor-Krull-base-change-graph} and Lemma \ref{lemma-Krull-beta}.
\end{proof}

\subsection{Directed cycles through a particular vertex or edge} We prove Theorem \ref{thm-KR-detect-vertex} and Corollary \ref{cor-KR-detect-edge} in this subsection. First, we state several simple algebraic lemmas that will be used in the proof.

\begin{lemma}\label{lemma-deloop}
Let $X$, $Y$ and $Z$ be three pairwise disjoint finite sets of variables. Then the sets $\Delta:=\{e_l(X\cup Z) - e_l(Y \cup Z) ~|~ l>0\}$ and $\Delta':=\{e_l(X) - e_l(Y) ~|~ l>0\}$ generate the same ideal of $\zed[X\cup Y \cup Z]$, where $e_l(S)$ is the $l$-th elementary symmetric polynomial in the finite set $S$ of variables.

In particular, for any directed graph $G$, the incidence ideal $I(G)$ of $\zed[E(G)]$ admits a set of generators in $\zed[E(G)\setminus L(G)]$, where $L(G)$ is the set of loop edges of $G$.
\end{lemma}

\begin{proof}
Let $I$ (resp. $I'$) be the ideal of $\zed[X\cup Y \cup Z]$ generated by $\Delta$ (resp. $\Delta'$.) For any $l>0$, 
\[
e_l(X\cup Z) - e_l(Y \cup Z) = \sum_{i=0}^l e_{l-i}(Z)(e_i(X)-e_i(Y)) \in I'.
\]
So $I \subset I'$. On the other hand,
\[
e_l(X) - e_l(Y) = \sum_{i=0}^l (-1)^{l-i} h_{l-i}(Z) (e_i(X\cup Z) - e_i(Y \cup Z)) \in I,
\]
where $h_j(Z)$ is the $j$-th complete symmetric polynomial in $Z$. So $I'\subset I$.
\end{proof}

\begin{lemma}\label{lemma-separate}
Let $X$ and $Y$ be two disjoint finite sets of variables. Assume $f_1,\dots,f_n$ are homogeneous polynomials in $\zed[X]$ of positive degrees, and $g_1,\dots,g_m$ are homogeneous polynomials in $\zed[Y]$ of positive degrees. Denote by $I$ the ideal $I= (f_1,\dots,f_n,g_1,\dots,g_m)$ of $\zed[X\cup Y]$. Then $\ann_{\zed[X]} (\zed[X\cup Y]/I)$ is the ideal of $\zed[X]$ generated by $f_1,\dots,f_n$.
\end{lemma}

\begin{proof}
Denote by $\hat{I}$ the ideal of $\zed[X]$ generated by $f_1,\dots,f_n$. Clearly, $\ann_{\zed[X]} (\zed[X\cup Y]/I) = I \cap \zed[X]$. In particular, $f_1,\dots,f_n \in \ann_{\zed[X]} (\zed[X\cup Y]/I)$ and, therefore, $\hat{I} \subset \ann_{\zed[X]} (\zed[X\cup Y]/I)$. Now assume $\varphi \in \ann_{\zed[X]} (\zed[X\cup Y]/I)=I \cap \zed[X]$. Then there are $a_1,\dots,a_n,b_1,\dots,b_m \in \zed[X\cup Y]$ such that 
\[
\varphi = \sum_{i=1}^n a_i f_i + \sum_{j=1}^m b_jg_j.
\]
Substituting every variable in $Y$ by $0$, we get $\varphi|_{Y=0} =\varphi$, $f_i|_{Y=0} =f_i$, $g_j|_{Y=0}=0$ and $a_i|_{Y=0} \in \zed[X]$. Thus,
\[
\varphi = \varphi|_{Y=0} = \sum_{i=1}^n (a_i|_{Y=0}) f_i \in \hat{I}.
\]
This shows that $\ann_{\zed[X]} (\zed[X\cup Y]/I) \subset \hat{I}$.
\end{proof}

\begin{lemma}\label{lemma-fg-restriction}
Let $X$ and $Y$ be two disjoint finite sets of variables. Assume $I$ is an ideal of $\zed[X \cup Y]$ such that $M:= \zed[X \cup Y]/I$ is a finitely generated $\zed$-module. Then $\zed[X]/(\ann_{\zed[X]} M) = \zed[X]/(\zed[X]\cap I)$ is also a finitely generated $\zed$-module.
\end{lemma}

\begin{proof}
The standard inclusion $\zed[X]\hookrightarrow \zed[X \cup Y]$ induces a well-defined injective $\zed[X]$-module homomorphism $\zed[X]/(\ann_{\zed[X]} M) =\zed[X]/(\zed[X]\cap I) \hookrightarrow M= \zed[X \cup Y]/I$. Since $M$ is a finitely generated $\zed$-module and $\zed$ is a Noetherian ring, this implies that $\zed[X]/(\ann_{\zed[X]} M)$ is also a finitely generated $\zed$-module.
\end{proof}

Now we are ready to work on Theorem \ref{thm-KR-detect-vertex}.

\begin{lemma}\label{lemma-fg-KR-vertex}
Let $G$ be a directed graph, and $v$ a vertex in $G$. Denote by $E(v)$ the set of all edges of $G$ incident at $v$. If there are no directed cycles in $G$ containing $v$, then $\zed[E(v)]/\ann_{\zed[E(v)]}(\mathscr{H}_0 (G))$ is a finitely generated $\zed$-module.
\end{lemma}

\begin{proof}
Denote by $\hat{G}$ the connected component of $G$ containing $v$. That is, $\hat{G}$ is the subgraph of $G$ given by
\begin{enumerate}
	\item $V(\hat{G}) = \{u \in V(G) ~|$ there is an undirected path in $G$ from $u$ to $v\}$,
	\item $E(\hat{G}) = \{x \in E(G)~|~ x$ is incident at some $u\in V(\hat{G})\}$.
\end{enumerate}
Consider the following subsets of $V(\hat{G})$. 
\begin{itemize}
	\item $S=\{ s \in V(G) ~|~s\neq v,$ there is a directed path from $s$ to $v\}$,
	\item $T=\{ t \in V(G) ~|~t\neq v,$ there is a directed path from $v$ to $t\}$,
	\item $U = V(\hat{G}) \setminus (S \cup T \cup \{v\})$.
\end{itemize}
Since there are no directed cycles in $G$ containing $v$, $S$ and $T$ are disjoint. So $\{S,T,U,\{v\}\}$ is a partition of $V(\hat{G})$. For $A,B \in \{S,T,U,\{v\}\}$, denote by $E_{A\rightarrow B}$ the set 
\[
E_{A\rightarrow B} = \{x \in E(G)~|~ \text{the initial vertex of } x \text{ is in } A, \text{ and the terminal vertex of } x \text{ is in } B\}.
\]  
Since there are no directed cycles in $G$ containing $v$, $E_{T\rightarrow S} = \emptyset$. By the definitions of $S,~T,~U$, we also have $E_{T\rightarrow U}=E_{U\rightarrow S} = E_{\{v\}\rightarrow U} = E_{U\rightarrow \{v\}} = \emptyset$.

Now define a directed graph $G'$ from $G$ by
\begin{enumerate}
	\item identifying all vertices in $S$ into a single vertex, called $s$,
	\item identifying all vertices in $T$ into a single vertex, called $t$,
	\item identifying all vertices in $U$ into a single vertex, called $u$.
\end{enumerate}
Note that: 
\begin{itemize}
	\item There are no directed cycles in $G'$ containing $v$.
	\item There is a natural one-to-one correspondence between edges of $G$ and edges of $G'$.
\end{itemize}

We identify edges of $G$ and edges of $G'$ by this correspondence. Then, by Lemma \ref{lemma-fuse-vertex}, there is a surjective $\zed[E(G)]$-module homomorphism $\mathscr{H}_0(G') \rightarrow \mathscr{H}_0(G)$. 

The connected component of $G'$ containing $v$ has exactly four vertices: $v$, $s$, $t$ and $u$. As sets of edges in $G'$, $E_{S \rightarrow S}$, $E_{T \rightarrow T}$ and $E_{U \rightarrow U}$ are the sets of loops at $s$, $t$ and $u$, respectively. Define
\[
E = E_{S \rightarrow \{v\}} \cup E_{S \rightarrow T} \cup E_{S \rightarrow U} \cup E_{\{v\} \rightarrow T} \cup E_{U \rightarrow T}.
\]
Then, as a set of edges in $G'$, $E$ is the set of non-loop edges in the connected component of $G'$ containing $v$. Define a subgraph $G''$ of $G'$ by $V(G'') = \{v,s,t,u\}$ and $E(G'')=E$. Then $G''$ contains no directed cycles.

Consider the incidence ideal $I(G')$ for $G'$ of $\zed[E(G')]=\zed[E(G)]$ generated by the incidence relations $\Delta_{G'}$ of $G'$. Applying Lemma \ref{lemma-deloop} at $v,s,t,u$, one gets generators of $I(G')$ in ${\zed[E(G)\setminus (E_{S \rightarrow S} \cup E_{T \rightarrow T} \cup E_{U \rightarrow U})]}$ such that:
\begin{enumerate}
	\item Each of these generators is a homogeneous polynomial of positive degree in either $\zed[E]$ or $\zed[E(G)\setminus (E_{S \rightarrow S} \cup E_{T \rightarrow T} \cup E_{U \rightarrow U} \cup E)]$.
	\item The generators in $\zed[E]$ generate the incidence ideal $I(G'')$ for $G''$ of $\zed[E]$. (In fact, these generators are exactly the incidence relations of $G''$.)
\end{enumerate}
Thus, by Lemma \ref{lemma-separate}, $\ann_{\zed[E]} \mathscr{H}_0 (G') = \ann_{\zed[E]} (\zed[E(G)]/I(G')) = I(G'')$. Since $G''$ contains no directed cycles, by Theorem \ref{thm-KR-detect}, $\zed[E]/(\ann_{\zed[E]} \mathscr{H}_0 (G')) = \zed[E]/I(G'') = \mathscr{H}_0 (G'')$ is a finitely generated $\zed$-module. Note that $E(v)=E_{S \rightarrow \{v\}} \cup E_{\{v\} \rightarrow T}$ is a subset of $E$. Using Lemma \ref{lemma-fg-restriction}, one has that $\zed[E(v)]/(\ann_{\zed[E(v)]} \mathscr{H}_0 (G')) = \zed[E(v)]/ (\zed[E(v)] \cap \ann_{\zed[E]} \mathscr{H}_0 (G'))$ is also a finitely generated $\zed$-module.

Recall that there is a surjective $\zed[E(G)]$-module homomorphism $\mathscr{H}_0(G') \rightarrow \mathscr{H}_0(G)$. This implies that $\ann_{\zed[E(v)]} \mathscr{H}_0 (G') \subset \ann_{\zed[E(v)]} \mathscr{H}_0 (G)$. Therefore, $\zed[E(v)]/(\ann_{\zed[E(v)]} \mathscr{H}_0 (G))$ is a finitely generated $\zed$-module too.
\end{proof}

\begin{lemma}\label{lemma-Krull-beta-vertex}
Let $G$ be a directed graph, and $v$ a vertex in $G$. Denote by $E(v)$ the set of all edges of $G$ incident at $v$. Then $\Kdim_{\zed[E(v)]} \mathscr{H}_0 (G) \leq \beta_v(G)+1$.
\end{lemma}

\begin{proof}
Let $k = \beta_v(G)$ and $\{x_1,\dots,x_k\}$ a subset of $E(v)$ whose removal destroys all directed cycles in $G$ containing $v$. Denote by $G'$ the directed graph obtained from $G$ by removing the edges $x_1,\dots,x_k$. Then there are no directed cycles in $G'$ containing $v$. Denote by $E_{G'}(v)$ the edges in $G'$ incident at $v$, that is, $E_{G'}(v)=E(v)\setminus \{x_1,\dots,x_k\}$. By Lemma \ref{lemma-edge-removal}, $\mathscr{H}_0 (G')$ is the $\zed[E(G)]$-module $\mathscr{H}_0 (G)/(I\cdot \mathscr{H}_0 (G))$, where $I$ is the ideal of $\zed[E(G)]$ generated by $x_1,\dots,x_k$. Since $\{x_1,\dots,x_k\} \subset E(v)$, one can see that $\ann_{\zed[E(v)]} \mathscr{H}_0 (G') = \ann_{\zed[E(v)]} \mathscr{H}_0 (G) + I_v$, where $I_v$ is the ideal of $\zed[E(v)]$ generated by $x_1,\dots,x_k$. Therefore, $\zed[E_{G'}(v)]/\ann_{\zed[E_{G'}(v)]} \mathscr{H}_0 (G') \cong (\zed[E(v)]/\ann_{\zed[E(v)]} \mathscr{H}_0 (G))/(x_1,\dots,x_k)$, where $(x_1,\dots,x_k)$ is the ideal of $\zed[E(v)]/\ann_{\zed[E(v)]} \mathscr{H}_0 (G)$ generated by $x_1,\dots,x_k$. But, by Lemma \ref{lemma-fg-KR-vertex}, $\zed[E_{G'}(v)]/\ann_{\zed[E_{G'}(v)]} \mathscr{H}_0 (G')$ is a finitely generated $\zed$-module. So, by Lemmas \ref{lemma-Krull-quotient}, \ref{lemma-fg-extension}, \ref{lemma-Krull-polynomial} and \ref{lemma-fg},
\[
\Kdim_{\zed[E(v)]} \mathscr{H}_0 (G) = \Kdim \zed[E(v)]/\ann_{\zed[E(v)]} \mathscr{H}_0 (G) = \Kdim R \leq \Kdim \zed[x_1,\dots,x_k] = k+1,
\]
where $R$ is the $\zed$-subalgebra of $\zed[E(v)]/\ann_{\zed[E(v)]} \mathscr{H}_0 (G)$ generated by $x_1,\dots,x_k$. This proves the lemma.
\end{proof}

We have effectively proved Theorem \ref{thm-KR-detect-vertex}. Here we recap what was done.

\begin{proof}[Proof of Theorem \ref{thm-KR-detect-vertex}]
In Part 1, (1) $\Rightarrow$ (2) follows from Lemma \ref{lemma-fg-KR-vertex}. (2) $\Leftrightarrow$ (3) follows from Lemma \ref{lemma-KR-module}. (2) $\Leftrightarrow$ (4) is elementary. (2) $\Rightarrow$ (5) follows from Lemma \ref{lemma-fg-extension}. (5) $\Rightarrow$ (6) follows from Corollary \ref{cor-Krull-base-change-graph}. (6) $\Rightarrow$ (1) follows from Proposition \ref{prop-Krull-bound-E}. This completes the proof of Part 1.

Part 2 follows from Proposition \ref{prop-Krull-bound-E}, Corollary \ref{cor-Krull-base-change-graph} and Lemma \ref{lemma-Krull-beta-vertex}.
\end{proof}

Corollary \ref{cor-KR-detect-edge} follows easily from Theorem \ref{thm-KR-detect-vertex}.

\begin{proof}[Proof of Corollary \ref{cor-KR-detect-edge}]
Denote by $\widetilde{G}$ the directed graph obtained from $G$ by adding a new vertex $v$ in the middle of $x$. This splits $x$ into two new edges, whose directions are given by the direction $x$. Call one of these new edges $x$ and the other $y$. Then $E(\widetilde{G}) = E(G) \cup \{y\}$. And the incidence ideal $I(\widetilde{G})$ for $\widetilde{G}$ is the ideal of $\zed[E(\widetilde{G})]$ generated by $x-y$ and all the incidence relations of $G$. Therefore, by Lemma \ref{lemma-KR-module}, 
\[
\zed[x,y]/\ann_{\zed[x,y]}(\mathscr{H}_\ast (\widetilde{G})) \cong \zed[x,y]/\ann_{\zed[x,y]}(\mathscr{H}_0 (\widetilde{G})) \cong \zed[x]/\ann_{\zed[x]}(\mathscr{H}_0 (G)).
\] 
Similarly, 
\[
\Q[x,y]/\ann_{\Q[x,y]}(\mathscr{H}_0^\Q (\widetilde{G})) \cong \Q[x]/\ann_{\Q[x]}(\mathscr{H}_0^\Q (G)).
\]
Note that:
\begin{itemize}
	\item $\{x,y\}$ is the set of edges of $\widetilde{G}$ incident at $v$.
	\item There is a one-to-one correspondence between directed cycles in $G$ containing $x$ and directed cycles in $\widetilde{G}$ containing $v$. 
\end{itemize}
Thus, we get Corollary \ref{cor-KR-detect-edge} by applying Theorem \ref{thm-KR-detect-vertex} to the vertex $v$ of $\widetilde{G}$.
\end{proof}

\section{Undirected Cycles}\label{sec-cycles-u}

\subsection{$\mathscr{H}_\ast$ and undirected cycles} Next we prove Theorems \ref{thm-tree}, \ref{thm-tree-v} and Corollary \ref{cor-u-cycle-detect}.

\begin{proof}[Proof of Theorem \ref{thm-tree}]
Note that the graded $\zed[E(G)]$-module $\zed$ in Part 1 of Theorem \ref{thm-tree} is supported on module grading $0$. So, to prove Parts 1 and 2 of Theorem \ref{thm-tree}, we just need to verify that:
\begin{enumerate}[(a)]
	\item The underlying undirected graph of $G$ is a disjoint union of trees $\Rightarrow$ $\mathscr{H}_\ast(G) \cong \mathscr{H}_0(G) \cong \zed$.
	\item $G$ contains an undirected cycle $\Rightarrow$ $\mathscr{H}_{0,1} (G) \neq 0$.
\end{enumerate}

To prove (a), we induct on the number of edges in $G$. If $G$ has only one edge $x$, then 
\[
\mathscr{C}_\ast (G) = 0 \rightarrow \zed[x]\{1\} \xrightarrow{x} \zed[x] \rightarrow 0.
\] 
(a) is obvious in this case. Assume that (a) is true for disjoint unions of trees with $n-1$ edges and that $G$ is a disjoint union of trees with $n$ edges. $G$ has a vertex of degree $1$. Applying Lemma \ref{lemma-end-removal} to this vertex, we get a disjoint union of trees $\hat{G}$ with $n-1$ edges such that $\mathscr{H}_\ast (G) \cong \mathscr{H}_\ast (\hat{G})$. This completes the induction and proves (a).

For (b), fix an undirected cycle $C$ in $G$. Applying Lemma \ref{lemma-edge-removal} to the removal of all edges of $G$ not in $C$, we get a surjective homomorphism $\mathscr{H}_{0} (G) \rightarrow \mathscr{H}_{0} (C)$ preserving the module grading. 

We call a vertex of $C$ removable if it is the terminal vertex of an edge in $C$ and the initial vertex of an edge in $C$. A vertex that is not removable is called non-removable. Note that there are even number of non-removable vertices in $C$ since non-removable vertices are alternatingly sources and sinks in $C$.

Next, we successively remove each removable vertex of $C$ and merging the two edges incident at it. We stop when we reach a graph $C'$ that is 
\begin{enumerate}
	\item either a directed cycle with a single removable vertex and a single loop edge,
	\item or an undirected cycle with no removable vertices and a positive even number of non-removable vertices.
\end{enumerate}
By Lemma \ref{lemma-vertex-2-removal},
\[
\mathscr{H}_{0} (C) \cong \mathscr{H}_{0} (C') \cong 
\begin{cases}
\zed[x] & \text{in case (1),} \\
\zed[x]/(x^2) & \text{in case (2).}
\end{cases}
\]
In both cases, $\mathscr{H}_{0,1} (C') \neq 0$. Then the aforementioned surjective homomorphism $\mathscr{H}_{0} (G) \rightarrow \mathscr{H}_{0} (C) \cong \mathscr{H}_{0,1} (C')$ preserving the module grading implies that $\mathscr{H}_{0,1} (G) \neq 0$. This proves (b) and completes the proof of Parts 1 and 2 of Theorem \ref{thm-tree}.

Next, we prove Part 3 of Theorem \ref{thm-tree}. Let $\alpha=\alpha_{undirected}(G)$. Fix a collection $\{C_1,\dots,C_\alpha\}$ of pairwise edge-disjoint undirected cycles in $G$ and a collection of edges $\{x_1,\dots,x_\alpha\} \subset E(G)$ such that $x_i$ is contained in $C_i$. Note that $x_i \neq x_j$ if $i\neq j$ since the collection $\{C_1,\dots,C_{\alpha}\}$ is pairwise edge-disjoint. Denote by $G'$ the subgraph of $G$ consisting of the vertices and edges of the undirected cycles in this collection. Then, by Lemma \ref{lemma-edge-removal}, there is a surjective homomorphism $\mathscr{H}_{0} (G) \rightarrow \mathscr{H}_{0} (G')$ preserving the module grading. For each vertex of degree $2k$ of $G'$ split it into $k$ vertices of degree $2$ such that, at each of these new vertices, there are two edges of the same undirected cycle in the above collection. This gives a new directed graph $G''$ that is the disjoint union of the undirected cycles $C_1,\dots,C_\alpha$. By Lemma \ref{lemma-fuse-vertex}, there is a surjective homomorphism $\mathscr{H}_{0} (G') \rightarrow \mathscr{H}_{0} (G'')$ preserving the module grading. Similar to the discussion in the proof of (b) above, we get
\[
\mathscr{H}_{0} (C_i) \cong
\begin{cases}
\zed[x_i] & \text{if all vertices of } C_i \text{ are removable,} \\
\zed[x_i]/(x_i^2) & \text{if not all vertices of } C_i \text{ are removable.} \\
\end{cases}
\]
But $\mathscr{H}_{0} (G'') \cong  \mathscr{H}_{0} (C_1)\otimes_{\zed} \cdots\otimes_{\zed} \mathscr{H}_{0} (C_\alpha)$. So, over $\zed$, $\mathscr{H}_{0,1} (G'')$ has rank $\alpha$ and is spanned by $x_1,\dots,x_\alpha$. The surjective homomorphisms $\mathscr{H}_{0} (G) \rightarrow \mathscr{H}_{0} (G')\rightarrow \mathscr{H}_{0} (G'')$ preserving the module grading then imply that $\rank \mathscr{H}_{0,1} (G) \geq \rank \mathscr{H}_{0,1} (G') \geq \rank \mathscr{H}_{0,1} (G'') = \alpha$. 

It remains to show that $\rank \mathscr{H}_{0,1} (G) \leq \beta_{undirected}(G)$. Let $\beta = \beta_{undirected}(G)$. Fix a set $\{y_1,\dots,y_\beta\}$ of edges of $G$ whose removal from $G$ destroys all undirected cycles in $G$. Denote by $\hat{G}$ the directed graph obtained from $G$ by removing the edges $y_1,\dots,y_\beta$. Then $\hat{G}$ contains no undirected cycles. By Lemma \ref{lemma-edge-removal}, $\mathscr{H}_{0} (\hat{G}) \cong \mathscr{H}_{0} (G)/((y_1,\dots,y_\beta)\cdot\mathscr{H}_{0} (G))$, where $(y_1,\dots,y_\beta)$ is the ideal of $\zed[E(G)]$ generated by $y_1,\dots,y_\beta$. So $\mathscr{H}_{0,1} (\hat{G}) \cong \mathscr{H}_{0,1} (G)/\left\langle y_1,\dots,y_\beta\right\rangle$, where $\left\langle y_1,\dots,y_\beta\right\rangle$ is the $\zed$-submodule of $\mathscr{H}_{0,1} (G)$ spanned by $y_1,\dots,y_\beta$. But, by Part 1, $\mathscr{H}_{0,1} (\hat{G})=0$. So $\mathscr{H}_{0,1} (G) = \left\langle y_1,\dots,y_\beta\right\rangle$ and, therefore, $\rank \mathscr{H}_{0,1} (G) \leq \beta$. This completes the proof of Part 3.
\end{proof}

\begin{proof}[Proof of Theorem \ref{thm-tree-v}]
Similar to the proof of Theorem \ref{thm-tree}, to prove Parts 1 and 2, we only need to show that
\begin{enumerate}[(a)]
	\item There are no undirected cycles in $G$ containing $v$ $\Rightarrow$ $A(v) \cong \zed$.
	\item There is an undirected cycle in $G$ containing $v$ $\Rightarrow$ $A_1(v) \neq 0$.
\end{enumerate}

Assume that there are no undirected cycles in $G$ containing $v$. Denote by $G_0$ the connected component of $G$ containing $v$. Consider the graph $G_0'$ given by $V(G_0')= V(G_0)$ and $E(G_0')= E(G_0)\setminus E(v)$. Since there are no undirected cycles in $G$ containing $v$, there is exactly one edge in $G$ connecting $v$ and each connected component of $G_0'$ that is not $v$ itself. Now define a graph $\hat{G}_0$ by $E(\hat{G}_0) = E(G_0)$ and $V(\hat{G}_0) = V(G_0)/\sim$, where $\sim$ is the equivalence relation $u \sim w$ if and only if $u$ and $w$ are in the same connected component of $G_0'$. Then $\hat{G}_0$ satisfies:
\begin{itemize}
	\item If $u \in V(\hat{G}_0)$ and $u \neq v$, then there is a single edge in $\hat{G}_0$ connecting $u$ and $v$.
	\item If $u,w \in V(\hat{G}_0)$ and $u \neq v$, $w \neq v$, $u \neq w$, then there are no edges in $\hat{G}_0$ connecting $u$ and $w$.
	\item There are no loops in $\hat{G}_0$ at $v$.
\end{itemize}
Let $\hat{G} = (G \setminus G_0) \sqcup \hat{G}_0$. By Lemmas \ref{lemma-deloop} and \ref{lemma-separate}, $\ann_{\zed[E(v)]} \mathscr{H}_0 (\hat{G})=(E(v))$, where $(E(v))$ is the homogeneous ideal of $\zed[E(v)]$ generated by $E(v)$. Thus, $\zed[E(v)]/(\ann_{\zed[E(v)]} \mathscr{H}_0 (\hat{G})) \cong \zed$. But, by Lemma \ref{lemma-fuse-vertex}, there is a surjective $\zed[E(G)]$-module homomorphism $\mathscr{H}_0 (\hat{G}) \rightarrow \mathscr{H}_0 (G)$. So, by Lemma \ref{lemma-Krull-quotient}, 
\[
(E(v)) = \ann_{\zed[E(v)]} \mathscr{H}_0 (\hat{G}) \subset \ann_{\zed[E(v)]} \mathscr{H}_0 (G) = \zed[E(v)]\cap I(G) \subset (E(v)),
\]
where $I(G)$ is the incidence ideal for $G$ of $\zed[E(G)]$. Therefore, $\ann_{\zed[E(v)]} \mathscr{H}_0 (G) = (E(v))$ and 
\[
\zed[E(v)]/(\ann_{\zed[E(v)]} \mathscr{H}_0 (G)) \cong \zed.
\] 
This proves (a).

Now assume that there is an undirected cycle in $G$ containing $v$. Then $\alpha_v:=\alpha_{undirected}(G,v) >0$. Fix a collection of $\{C_1,\dots,C_{\alpha_v}\}$ of pairwise edge-disjoint undirected cycles in $G$ containing $v$. For each $C_i$, pick an $x_i \in E(v)$ contained in $C_i$. Note that $x_i \neq x_j$ if $i\neq j$ since the collection $\{C_1,\dots,C_{\alpha_v}\}$ is pairwise edge-disjoint. Let $G'$ be the graph given by 
\begin{itemize}
	\item $V(G')=\{u\in V(G)~|~ u$ is contained in $C_i$ for some $i=1,\dots,\alpha_v\}$,
	\item $E(G')=\{x\in E(G)~|~ x$ is contained in $C_i$ for some $i=1,\dots,\alpha_v\}$.
\end{itemize}
For each vertex $u$ in $G'$ of degree $2k$ in $G'$, split it into $k$ vertices, each of which is the vertex where two edges of the same $C_i$ are incident. This changes $G'$ to another graph $G''$. Similar to the computation in the proof of Theorem \ref{thm-tree}, we have that 
\[
\mathscr{H}_{0} (C_i) \cong
\begin{cases}
\zed[x_i] & \text{if all vertices of } C_i \text{ are removable,} \\
\zed[x_i]/(x_i^2) & \text{if not all vertices of } C_i \text{ are removable,} \\
\end{cases}
\]
and $\mathscr{H}_{0} (G'') \cong  \mathscr{H}_{0} (C_1)\otimes_{\zed} \cdots\otimes_{\zed} \mathscr{H}_{0} (C_{\alpha_v})$. 

If $x\in E(v)$ is an edge of a $C_i$, then $i$ is uniquely determined by $x$ since $C_1,\dots,C_{\alpha_v}$ are pairwise edge-disjoint. Define $\nu(x)$ to be the parity of the number of non-removable vertices of $C_i$ between $x$ and $x_i$, which is well-defined since the total number of non-removable vertices on $C_i$ is even. Using this notation, $\ann_{\zed[E(v)]} (\mathscr{H}_{0} (G''))$ is the ideal of $\zed[E(v)]$ generated by the set $(E(v)\setminus E(G'')) \cup S_1 \cup S_2$, where 
\begin{eqnarray*}
S_1 & = & \{x_i-(-1)^{\nu(x)}x~|~x\in E(v) \cap E(G''),~1\leq i \leq \alpha_v \text{ and } x \text{ is contained in } C_i\}, \\
S_2 & = & \{x_i^2~|~ 1\leq i \leq \alpha_v \text{ and } C_i \text{ contains a non-removable vertex}\},
\end{eqnarray*}
This shows that $\zed[E(v)]/ \ann_{\zed[E(v)]} (\mathscr{H}_{0} (G''))\cong \zed[x_1,\dots,x_{\alpha_v}]/ I_2$, where $I_2$ is the ideal of $\zed[x_1,\dots,x_{\alpha_v}]$ generated by $S_2$. But, by Lemmas \ref{lemma-edge-removal} and \ref{lemma-fuse-vertex}, there are surjective $\zed[E(G)]$-module maps $\mathscr{H}_{0} (G) \rightarrow \mathscr{H}_{0} (G') \rightarrow \mathscr{H}_{0} (G'')$. So, by Lemma \ref{lemma-Krull-quotient}, $\ann_{\zed[E(v)]}\mathscr{H}_{0} (G) \subset \ann_{\zed[E(v)]}\mathscr{H}_{0} (G') \subset \ann_{\zed[E(v)]}\mathscr{H}_{0} (G'')$. Thus, there is a surjective $\zed[E(G)]$-module map 
\[
A(v) =\zed[E(v)]/\ann_{\zed[E(v)]}(\mathscr{H}_{0} (G)) \rightarrow \zed[E(v)]/\ann_{\zed[E(v)]}(\mathscr{H}_{0} (G'')) \cong \zed[x_1,\dots,x_{\alpha_v}]/ I_2.
\]
In particular, it follows that $\rank A_1(v) \geq \alpha_v >0$. This proves (b) and that  $\rank A_1(v) \geq \alpha_{undirected}(G,v)$. 

So far, we have proved Parts 1, 2 and half of Part 3. 

It remains to show that $\rank A_1(v) \leq \beta_{undirected}(G,v)$. Write $\beta_v = \beta_{undirected}(G,v)$ and fix a set $\{y_1,\dots,y_{\beta_v}\}$ of $\beta_v$ edges in $G$ incident at $v$, whose removal destroys all undirected cycles in $G$ containing $v$. Let $\tilde{G}$ be the graph with $V(\tilde{G}) =V(G)$ and $E(\tilde{G})=E(G)\setminus \{y_1,\dots,y_{\beta_v}\}$. It is clear that:
\begin{itemize}
	\item There are no undirected cycles in $\tilde{G}$ containing $v$. 
	\item $\ann_{\zed[E(v)]} \mathscr{H}_{0} (\tilde{G}) = \ann_{\zed[E(v)]}\mathscr{H}_{0} (G) + (y_1,\dots,y_{\beta_v})$, where $(y_1,\dots,y_{\beta_v})$ is the ideal of $\zed[E(v)]$ generated by $\{y_1,\dots,y_{\beta_v}\}$.
\end{itemize}
Denote by $J$ the ideal of $A(v)$ generated by $\{y_1,\dots,y_{\beta_v}\}$. Then, by Part 1, 
\[
A(v)/J \cong \zed[E(V)]/(\ann_{\zed[E(v)]}\mathscr{H}_{0} (G) + (y_1,\dots,y_{\beta_v})) = \zed[E(V)]/\ann_{\zed[E(v)]} \mathscr{H}_{0} (\tilde{G}) \cong \zed,
\] 
which does not contain homogeneous elements of positive degrees. In particular, $A_1(v)/\left\langle y_1,\dots,y_{\beta_v}\right\rangle \cong 0$, where $\left\langle y_1,\dots,y_{\beta_v}\right\rangle$ is the $\zed$-submodule of $A_1(v)$ spanned by $\{y_1,\dots,y_{\beta_v}\}$. So $A_1(v)=\left\langle y_1,\dots,y_{\beta_v}\right\rangle$ and $\rank A_1(v) \leq \beta_v=\beta_{undirected}(G,v)$.
\end{proof}
 
\begin{proof}[Proof of Corollary \ref{cor-u-cycle-detect}]
Give each edge in $G$ a direction. This makes $G$ a directed graph. Recall that $\mathscr{H}_0(G) \cong \zed[E(G)]/I(G)$, where $I(G)$ is the incidence ideal of $\zed[E(G)]$. Note that $\delta_{v,l}$ is a homogeneous polynomial of degree $l$. So, as $\zed$-modules $\mathscr{H}_{0,1}(G) \cong \zed \cdot E(G)/M$, where $\zed \cdot E(G):= \bigoplus_{x\in E(G)} \zed \cdot x$ and $M$ is the $\zed$-submodule of $\zed \cdot E(G)$ generated by $\{\delta_{v,1}~|~ v \in V(G)\}$. In $\zed_2 \cdot E(G) \cong \zed \cdot E(G)/2(\zed \cdot E(G))$, $\delta_{v,1} = \sum_{x\in E(v)\setminus L(v)} x$. Thus, under the standard quotient map $\zed \cdot E(G) \rightarrow \zed_2 \cdot E(G)$, $M$ is mapped onto $S$. This implies that $\zed_2 \otimes_\zed \mathscr{H}_{0,1}(G) \cong \zed_2 \cdot E(G) /S$. So, by Theorem \ref{thm-tree}, $|E(G)|-\dim_{\zed_2} S=\dim_{\zed_2} (\zed_2 \cdot E(G) /S) = \dim_{\zed_2} \zed_2 \otimes_\zed \mathscr{H}_{0,1}(G) \geq \rank \mathscr{H}_{0,1}(G) \geq \alpha_{undirected}(G)$. Repeat the last paragraph of the proof of Theorem \ref{thm-tree} over $\zed_2$, one gets that $|E(G)|-\dim_{\zed_2} S=\dim_{\zed_2} (\zed_2 \cdot E(G) /S) \leq \beta_{undirected}(G)$. This proves Part 1 of Corollary \ref{cor-u-cycle-detect}.

Recall that $A(v) := \zed[E(v)]/(\ann_{\zed[E(v)]} \mathscr{H}_0 (G))$ in Theorem \ref{thm-tree-v}. So $A_1(v) \cong \zed \cdot E(v) / ((\zed \cdot E(v)) \cap M)$. This implies that $\zed_2 \otimes_\zed A_1(v) \cong \zed_2 \cdot E(v) / ((\zed_2 \cdot E(v)) \cap S)$. So, similar to Part 1, Part 2 of Corollary \ref{cor-u-cycle-detect} follows from Theorem \ref{thm-tree-v} and the last paragraph of its proof.

It remains to prove Part 3. First, assume $x$ is contained in an undirected cycle $C$ in $G$. Modulo $\zed_2\cdot E(G)$ by the subspace spanned by $\{y \in E(G)~|~ y \notin E(C)\}$. This give a quotient map from $\zed_2\cdot E(G)$ to $\zed_2\cdot E(C)$, which maps $S$ to the subspace $S(E)$ of $\zed_2\cdot E(C)$ spanned by $\{y+x ~|~ y \in E(C)\}$. Clearly, $(\zed_2\cdot E(C))/S(E) \cong \zed_2 \cdot x$. So $x \notin S(E)$ as a vector in $\zed_2\cdot E(C)$. Thus, $x \notin S$ as a vector in $\zed_2\cdot E(G)$.

Now assume that $x$ is not contained in any undirected cycles in $G$. Denote by $G'$ the undirected graph obtained from $G$ by removing the edge $x$. Since there are no undirected cycles in $G$ containing $x$, the two vertices of $x$ belong to two different connected components of $G'$. Let $G'_0$ be one of these two connected components. Then $x = \sum_{v \in V(G'_0)} \sum_{y \in E(v)\setminus L(v)} y$, where $E(v)\setminus L(v)$ is the set of non-loop edges in $G$ (not just those in $G'$ or $G'_0$) incident at $v$. This is because each $y$ in this sum is either $x$, which appears once in this sum, or a non-loop edge of $G'_0$, which appears twice. This shows that, as a vector in $\zed_2\cdot E(G)$, $x \in S$.
\end{proof}

\subsection{Properties of $\mathscr{U}_\ast$} Next, we state some basic properties of $\mathscr{U}_\ast$ and sketch a proof of Proposition \ref{prop-KR-undirected}.

\begin{lemma}\label{lemma-KR-module-u}
Let $G$ be an undirected graph.
\begin{enumerate}[1.]
	\item $\mathscr{U}_0 (G) \cong \zed[E(G)]/(\Omega_G)$, where $(\Omega_G)$ is the ideal of $\zed[E(G)]$ generated by the set $\Omega_G = \{e_l(E(v))~|~v\in V(G), ~1 \leq l \leq \deg v\}$.
	\item $\mathscr{U}_\ast (G)$ is a finitely generated $\zed[E(G)]/(\Omega_G)$-module.
	\item $\mathscr{U}_\ast (G)$ is a finitely generated $\zed$-module if and only if $\mathscr{U}_0 (G)$ is a finitely generated $\zed$-module.
\end{enumerate}
\end{lemma}

The proof of Lemma \ref{lemma-KR-module-u} is similar to that of Lemma \ref{lemma-KR-module}.

\begin{lemma}\label{lemma-edge-removal-u}
Let $G$ be an undirected graph and $\hat{G}$ the undirected graph obtained from $G$ by removing a single edge $x\in E(G)$ incident at $u,v \in V(G)$. Under the standard identification $\zed[E(\hat{G})]\cong\zed[E(G)]/(x)$, 
\[
U_\ast(\hat{G}) \otimes_{\zed[E(\hat{G})]} U \cong U_\ast(G)/x\cdot U_\ast(G)
\] 
as graded Koszul chain complexes over $\zed[E(\hat{G})]$, where $U$ is the graded Koszul chain complex 
\[
U := 
\begin{cases}
(0 \rightarrow \underbrace{\zed[E(\hat{G})]\{\deg u\}}_{1} \xrightarrow{0} \underbrace{\zed[E(\hat{G})]}_{0} \rightarrow 0) \otimes_{\zed[E(\hat{G})]} (0 \rightarrow \underbrace{\zed[E(\hat{G})]\{\deg v\}}_{1} \xrightarrow{0} \underbrace{\zed[E(\hat{G})]}_{0} \rightarrow 0) & \text{if } u \neq v, \\
(0 \rightarrow \underbrace{\zed[E(\hat{G})]\{\deg u\}}_{1} \xrightarrow{0} \underbrace{\zed[E(\hat{G})]}_{0} \rightarrow 0) \otimes_{\zed[E(\hat{G})]} (0 \rightarrow \underbrace{\zed[E(\hat{G})]\{\deg u -1\}}_{1} \xrightarrow{0} \underbrace{\zed[E(\hat{G})]}_{0} \rightarrow 0) & \text{if } u = v.
\end{cases}
\]
In particular, $\mathscr{U}_0 (\hat{G}) \cong \mathscr{U}_0 (G) /x \cdot \mathscr{U}_0 (G)$ as graded $\zed[E(G)]$-modules under the standard identification $\zed[E(\hat{G})]\cong\zed[E(G)]/(x)$.
\end{lemma}

The proof of Lemma \ref{lemma-edge-removal-u} is straightforward and similar to that of Lemma \ref{lemma-edge-removal}.

\begin{lemma}\label{lemma-end-removal-u}
Let $G$ be an undirected graph, $v$ a vertex of $G$ of degree $1$ and $x$ the edge incident at $v$. Denote by $\hat{G}$ the undirected graph obtained from $G$ by removing $v$ and $x$. Then, under the standard identification $\zed[E(G)]=\zed[E(\hat{G})][x]$, $U_\ast(G)$ and $U_\ast(\hat{G})$ are homotopic as chain complexes of graded $\zed[E(\hat{G})]$-modules.

In particular, $\mathscr{U}_\ast (G) \cong \mathscr{U}_\ast (\hat{G})$ as $\zed \oplus \zed$-graded $\zed[E(G)]$-modules under the standard identification $\zed[E(\hat{G})]\cong\zed[E(G)]/(x)$.
\end{lemma}

The proof of Lemma \ref{lemma-end-removal-u} is similar to that of Lemma \ref{lemma-end-removal}.

\begin{lemma}\label{lemma-fuse-vertex-u}
Let $G$ be an undirected graph, and $u$, $v$ two distinct vertices of $G$. Denote by $G_{u \# v}$ the undirected graph obtained from $G$ by identifying $u$ and $v$ into a single vertex. Note that there is a natural one-to-one correspondence between $E(G)$ and $E(G_{u \# v})$, which gives an identification $\zed[E(G)]=\zed[E(G_{u \# v})]$. Under this identification, there is a surjective $\zed[E(G)]$-module homomorphism $\mathscr{U}_0 (G_{u \# v}) \rightarrow \mathscr{U}_0 (G)$.
\end{lemma}

The proof of Lemma \ref{lemma-fuse-vertex-u} is similar to that of Lemma \ref{lemma-fuse-vertex}.

\begin{lemma}\label{lemma-fg-u}
Let $G$ be an undirected graph. Then $\mathscr{U}_0(G)$ is a finitely generated $\zed$-module.
\end{lemma}

The proof of Lemma \ref{lemma-fg-u} is similar to that of Corollary \ref{cor-fg-acyclic}.

With the above lemmas, we can sketch a proof of Proposition \ref{prop-KR-undirected}.

\begin{proof}[Sketch of a proof of Proposition \ref{prop-KR-undirected}]
Part 1 follows from Lemmas \ref{lemma-KR-module-u} and \ref{lemma-fg-u}. Proof of Parts 2-4 is extremely similar to that of Theorem \ref{thm-tree}. One just needs to replace Lemmas \ref{lemma-KR-module}-\ref{lemma-fuse-vertex} in that proof by Lemmas \ref{lemma-KR-module-u}-\ref{lemma-fuse-vertex-u}.
\end{proof}

\end{document}